\documentclass[reqno, 11pt]{amsart}
\usepackage{comment,amssymb,latexsym,upref,enumerate}
\usepackage[normalem]{ulem}
\usepackage{mathrsfs,color}
\usepackage{mathtools}
\usepackage{centernot}
\usepackage[colorlinks,linkcolor=blue,citecolor=blue]{hyperref}
\usepackage{todonotes}
\usepackage{mnotes}

\allowdisplaybreaks

\numberwithin{equation}{section}

\theoremstyle{plain}
\newtheorem{thm}{Theorem}[section]
\newtheorem{lem}[thm]{Lemma}
\newtheorem{prop}[thm]{Proposition}

\theoremstyle{definition}
\newtheorem{Def}[thm]{Definition}

\theoremstyle{remark}
\newtheorem{rem}[thm]{Remark}

\newcommand{\define}{:=}
\newcommand{\DIV}{{\rm{div}}}
\input Tdef.def

\newcommand{\trho}{\bar{\rho}}
\newcommand{\phFmet}{\bar{g}}

\newcommand{\tlapse}{\tilde{N}}
\newcommand{\tshift}{\tilde{X}}
\newcommand{\phTmet}{\tilde{g}}

\newcommand{\C}[1]{{#1}}
\newcommand{\E}{E^{p}}
\newcommand{\Ric}{\text{Ric}}
\newcommand{\Riem}{\text{Riem}}
\newcommand{\R}{\text{R}}
\newcommand{\Eco}{\tilde{E}}
\newcommand{\bad}{B}

\begin{document}

\title[The Stability of Relativistic Fluids in Linearly
Expanding Cosmologies]{The Stability of Relativistic Fluids in Linearly
Expanding Cosmologies}

%\author[]{}
%\address{}

\author[D.~Fajman]{David Fajman}
\address{University of Vienna}
\email{David.Fajman@univie.ac.at}

\author[M.~Ofner]{Maximilian Ofner}
\address{University of Vienna}
\email{maximilian.ofner@univie.ac.at}

\author[T.A.~Oliynyk]{Todd A. Oliynyk}
\address{School of Mathematical Sciences\\
9 Rainforest Walk\\
Monash University, VIC 3800\\ Australia}
\email{todd.oliynyk@monash.edu}

\author[Z.~Wyatt]{Zoe Wyatt}
\address{Department of Mathematics, King’s College London, Strand, London, WC2R 2LS, United Kingdom}
\email{zoe.wyatt@kcl.ac.uk}

\begin{abstract}
In this paper we study cosmological solutions to the Einstein--Euler equations. We first establish the future stability of nonlinear perturbations of a class of homogeneous solutions
to the relativistic Euler equations on fixed linearly expanding cosmological spacetimes with a linear equation of state $p=K \rho$  for the parameter values $K \in (0,1/3)$. This removes the restriction to irrotational perturbations in earlier work \cite{FOW:2021}, and relies on  a novel transformation of the fluid variables that is well-adapted to Fuchsian methods.  We then apply this new transformation to show the global regularity and stability of the Milne spacetime under the coupled Einstein--Euler equations, again with a linear equation of state $p=K \rho$, $K \in (0,1/3)$. Our proof requires a correction mechanism to account for the spatially curved geometry. 
In total, this is indicative that structure formation in cosmological fluid-filled spacetimes requires an epoch of decelerated expansion.\newline

\begin{center}
   \textit{This work is dedicated to the memory of Bernd Schmidt.} 
\end{center}
    
\end{abstract}

%\begin{abstract}
%We establish the future nonlinear stability of a large class of cosmological models as solutions to the Einstein-Euler system. We consider the case of a vanishing cosmological constant, which in particular implies that the expansion rate of the respective models is linear i.e. has zero acceleration. A key feature of our proof is the introduction of a novel transformation which allows us to treat irrotational fluid perturbations.
%\end{abstract}

\maketitle

\section{Introduction}

\subsection{The Einstein-Euler system}
In this paper, we consider the Einstein-relativistic Euler equations 
\begin{subequations}\label{Ein-PF}
\begin{align}
\Ric[\gb]_{\mu\nu}-\frac12\gb_{\mu\nu}\R[\gb]&=\Tb_{\mu\nu},\label{Ein}
\\\label{rel-Eul} 
\nablab_\mu \Tb^{\mu\nu}&=0,\\
 \Tb^{\mu\nu} &= (\rhob + \pb)\vb^\mu \vb^\nu + \pb \gb^{\mu\nu}.\label{EnMom}
\end{align}
\end{subequations}
Here $\rhob$ is the proper energy density of the fluid, $\pb$  the fluid pressure and $\vb^\mu$ is the fluid 4-velocity, which we assume is normalized by
\begin{equation} \label{vb-norm}
    \gb_{\mu\nu}\vb^\mu \vb^\nu = -1.
\end{equation}

The system \eqref{Ein-PF} models a four-dimensional fluid-filled spacetime $(\mathcal{M}, \gb)$, whose time-evolution is determined by the coupled interaction between spacetime evolution under the Einstein equations \eqref{Ein} and fluid propagation under the relativistic Euler equations \eqref{rel-Eul}.
The Einstein--relativistic Euler equations are particularly relevant in cosmology, where they are used to model the Universe on astrophysical scales (see e.g. \cite{Ch31,OS39}). For simplicity, we henceforth drop  ``relativistic" when referring to the relativistic Euler equations. 

From an analytical point of view, the PDE system \eqref{Ein-PF} presents challenging problems since both the Einstein equations and the Euler equations may  form singularities. Singularities arising in the spacetime can occur in the context of gravitational collapse to black holes or a Big Bang, while singularities in solutions of the Euler equations correspond to fluid shock formation \cite{Ch07}.

In the cosmological context, singularity formation for the coupled Einstein-Euler equations can be interpreted as the onset of structure formation in the large scale evolution of spacetime \cite{BiBuKa92}. Structure formation in cosmology describes the process by which regions with high matter densities emerge from an initially homogeneous matter distribution. Hence, it  describes the formation of structures such as stars, galaxies and galaxy clusters from an initially homogeneous fluid. 
%Due to the complexity of the dynamics of the problem it is still widely open from a rigorous point of view. There are, however, recent contributions from the mathematical study of the problem, which provide significant bounds on the cosmological parameters, which allow for the structure formation process to occur. We discuss those in the following. 
\subsection{Previous work}
To close the system \eqref{Ein-PF}, we  consider the linear, barotropic equation of state
\begin{equation} \label{eos-lin}
    \pb = K\rhob.
\end{equation}
The equation of state parameter $K$ is a constant and gives the fluid speed of sound via $c_s = \sqrt{K}$. On physical grounds, $0 \leq K \leq 1$, however the most common applications in cosmology use $K \in [0,1/3]$.  The case $K=0$ corresponds to a dust fluid, while $K=1/3$ corresponds to a radiation fluid.

In the present article we consider cosmological spacetimes with Lorentzian metrics of the form 
\begin{equation}\label{background-cosmology}
-d\tb^2+a(\tb)^2 g_0,
\end{equation}
where $(M,g_0)$ is a closed Riemannian 3-manifold without boundary. The increasing function $a(\tb)$ is the \emph{scale factor}. We say the geometry exhibits \textit{accelerated expansion} if $\ddot{a}>0$, \textit{linear expansion} if $a(\tb) = \tb$ and \textit{decelerated expansion} if $\ddot{a}<0$. Note that the direction of cosmological expansion corresponds to $a(\tb)\rightarrow \infty$ as $\tb\nearrow\infty$. 
If \eqref{background-cosmology} solves the Einstein equations, then it is a solution to the Friedman equations describing an isotropic and (locally) homogeneous cosmology. A particularly important example for the present paper is the \textit{Milne model} where $a(\tb)=\tb$ and $(M, g_0)$ is a negative Einstein space satisfying $\Ric[g_0]=-\frac29 g_0$.

%In a spacetime of the form $\eqref{background-cosmology}$ the spatial volume is expanding at the rate $a(t)^3$. 

\subsubsection{Fluid stabilisation from cosmological expansion}

On a fixed Minkowski spacetime, small perturbations of a class of homogeneous solutions to the Euler equations \eqref{rel-Eul} are known to form shocks in finite time \cite{Ch07}. This result was shown for a large class of equations of state including \eqref{eos-lin}. By contrast, on cosmological backgrounds such as \eqref{background-cosmology}, accelerated spacetime expansion  is known to suppress shock formation in fluids. This was first discovered in the Newtonian cosmological setting  for a class of exponentially expanding spacetimes in \cite{BRR94}. In the context of Einstein-Euler, stability of homogeneous solutions on exponentially expanding cosmological spacetimes was first studied in \cite{RoSp13},   with several later works \cite{Fr17,HaSp15,  LiuOliynyk:2018b,LiuOliynyk:2018a,LVK13,MarshallOliynyk:2022,Oliynyk16, Oliynyk:2021,Sp12}. For various other results on fluid stabilization in the regime of accelerated expansion\cnote{purple}{,} we refer to \cite{LeFlochWei:2021,MondalPuskar2022,Wei:2018}. We mention also earlier work   \cite{Ri08} concerning the stability of solutions to the Einstein equations undergoing accelerated expansion.  

To go below accelerated expansion, i.e. when $\dot{a}>0$ but $\ddot{a}\leq 0$, it is illuminating to first study the fluid behaviour on fixed metrics of the form \eqref{background-cosmology}. 
Fluid stabilization depends on the spacetime expansion rate,  the fluid parameter $K$, and the geometry and topology of the expansion-normalized spatial geometry $(M,g_0)$. Roughly speaking,  larger speeds of sound and slower expansion rates tend to facilitate singularity formation, while slower speeds of sound and fast expansion rates suppress it.

To be concrete\footnote{We emphasise though that \cite{Sp13} considers more general conditions on the scale factor than we outline in this paragraph.}, consider homogeneous solutions to the Euler equations \eqref{rel-Eul} on fixed cosmological spacetimes undergoing power law inflation, i.e.  $a(\tb)= \tb^{\alpha}$ for $\alpha>0$ with $(M,g_0)$  flat. Note that $\alpha>1$ corresponds to accelerated expansion. In the case of dust $K=0$,  \cite{Sp13} showed that small perturbations of the homogeneous fluid solutions are globally regular for all  $\alpha>1/2$. If $K\in(0,1/3)$, work by some of the present authors  showed that  homogeneous fluid solutions are globally regular under irrotational perturbations  for $\alpha=1$ \cite{FOfW:2021}. The case $\alpha>1$ without the irrotational restriction was shown in \cite{Sp13}. Finally, if $K=1/3$, then \cite{Sp13} remarkably showed that radiation fluids do not stabilize for linear expansion i.e. the fluid stabilises if $\alpha>1$, while shocks develop in finite time if $\alpha=1$. 

Moving next to the coupled Einstein-Euler equations \eqref{Ein-PF}, the only stability result below accelerated expansion is in the case of dust $K=0$ and linear expansion. More precisely, some of the present authors showed the stability of the Milne model as a solution to the Einstein-dust equations \cite{FOfW:2021}. 
We mention for context that the Milne model is known to be a stable solution solution to the Einstein vacuum
equations \cite{AnderssonMoncrief:2011} as well as a solution to several Einstein-matter models \cite{AF20, BFK19, Wang-KG, FW21, BarzegarFajman20pub}.

Note that the case of negative spatial curvature is the only known class of solutions where we can study the fluid dynamics with gravitational backreaction in the regime of linear expansion. For the other FLRW models the long-time dynamics are either recollapsing (in the case of positive spatial curvature) towards a big-crunch singularity or a matter-dominated decelerated expansion in the case of toroidal spatial topology with vanishing curvature \cite{Re08}.
Finally, we emphasise that \textit{without} the backreaction of the fluid on the spacetime, \cite{Sp13} showed dust stability  with \textit{de}celerated expansion.

\textit{Notation:} Our indexing convention is as follows: lower case Greek letters, e.g.~$\mu, \nu, \gamma$, will label spacetime coordinate indices that run
from $0$ to $3$ while lower case Latin letters, e.g.~$i, j, k$, will label  spatial coordinate indices that run from $1$ to $3$.

\subsection{Results and technical advances in the present paper}
This paper is roughly divided into three parts. 
In the first part, outlined in Section \ref{intro:transf},  we introduce a \emph{novel transformation} that allows us to treat fluids with non-vanishing rotation. In the second part, see Section \ref{intro:middlesec} below, we apply our transformation to show fluid stability with $K\in(0,1/3)$  on \emph{linearly expanding spacetimes}:

\begin{thm}\label{thm:intro1}
The canonical homogeneous fluid solutions to \eqref{rel-Eul} with a linear equation of state \eqref{eos-lin} and $K \in (0,1/3)$ are nonlinearly stable in the expanding direction of fixed linearly-expanding FLRW spacetimes of the form $$(\mathbb{R}\times\mathbb{T}^3, -d\bar{t}^2 + \bar{t}^2 \delta_{ij}dx^idx^j).$$ 
\end{thm}
 The significance of Theorem \ref{thm:intro1} is that it removes the restriction to irrotational fluid perturbations in  \cite{FOW:2021}.
Finally in the third part, see Section \ref{intro:finalsec} below, we consider a fully nonlinear problem including backreaction on spacetimes with negative spatial Einstein geometries. We prove the following theorem:

\begin{thm}\label{thm:intro2}
All four-dimensional FLRW spacetime models with compact spatial slices and negative spatial Einstein geometry are future stable solutions of the Einstein-Euler equations \eqref{Ein-PF} with linear equation of state \eqref{eos-lin} and $K\in (0,1/3)$. 
\end{thm}

\subsubsection{Transformation of the fluid variables}\label{intro:transf}
The key advance in this paper is a new version of the Fuchsian method, originally introduced in \cite{FOW:2021}, which is designed to deal with general fluids without the irrotationality restriction. As first identified in \cite{Oliynyk16}, the Euler equations \eqref{rel-Eul} can be written as a symmetric hyperbolic Fuchsian PDE system as 
$$ B^0 \partial_t U + B^k \partial_k U = H.$$
Note that the time function $t$ appearing here is related to the cosmological time function $ \tb $  by $ t=1/\tb $, and hence timelike infinity is located at $ t\to 0 $. Precise definitions are given in Section \ref{sec:transf}, however the main idea is that the unknown $U = (\zeta, u^i)^{\tr}$ encodes the fluid energy density and velocity. We then transform to new unknowns $Z= (\psi,z^j)^{\tr}$ by
$$
(\zeta, u^i)=\bigl( a(\psi,|z|^2_{\gc}), b(\psi)z^i\bigr).$$
  The new variables obey the following PDE:
\begin{equation}\label{eq:intro_mod_E_EoM}
	A^0\del{t}Z + \frac{1}{t}A^k \Dc_k Z = H',
\end{equation}
for $H'$ is the source term as given in Lemma \ref{lem:transf-Eul-general}. 
By design, the functions $a, b$ are chosen so that the  matrices $A^0, A^k$  satisfy the relations
\begin{equation}\label{A-relations}
|\Pbb A^0 \Pbb^\perp|_{\op} = |\Pbb^\perp A^0 \Pbb|_{\op} =  \Ord\bigl(|\beta|_{\gc} +|z|_{\gc}^2\bigr), \qquad
   |\Pbb^\perp A^k \Pbb^\perp |_{\op} =  \Ord\bigl(|\beta|_{\gc} +|z|_{\gc}^2\bigr).
\end{equation}
In these expression $ \beta $ represents the shift vector of the dynamical metric and $ \Pbb $ and $ \Pbb^{\perp} $ are projections that single out specific matrix elements.  For a fixed expanding FLRW background metric, or small perturbations thereof, the geometric quantities in $ H^{\prime} $ are negligible and thus the source term can be written as
\begin{equation*}
        H'=
        c(K)\Pbb Z + \text{error}.
    \end{equation*}
Here $c(K)$ is a strictly positive constant, thanks to our choice of equation of state and $K\in(0,1/3)$, and the error is of order $ |z|_{\gc}^{2} $.
Extracting this structure on both sides of the equation \eqref{eq:intro_mod_E_EoM} is crucial for the analysis of the system. For further details we refer to Remark \ref{rem:conjeffect}.

\subsubsection{Application to fluids on linearly expanding spacetimes}\label{intro:middlesec}
In Section \ref{section:fixed_MFLRW} we apply our novel transformation to establish small data stability for the homogeneous solutions to the Euler equation on the fixed Milne-like (i.e.~linearly expanding spatially flat) FLRW background. After the application of the transformation described above, we show, due to \eqref{A-relations} as well as the decomposition of $ H^{\prime}$, that the conditions of the Fuchsian global existence theorem from \cite[Thm.~4.5]{FOW:2021} are satisfied. Note that 
Theorem \ref{thm:intro1} is also true for  linearly expanding FLRW-type spacetimes $ (\mathbb{R}\times M,-d\bar{t}^{2}+\bar{t}^{2}g_0) $, where $ (M,g_0) $ is an arbitrary closed Riemannian manifold. This extension is discussed in Remark \ref{arbitrarygeometry}.
Finally, we expect that our method is robust enough to  be applied to other spacetime geometries and expansion rates than considered here.

\subsubsection{Application to Einstein-fluid spacetimes with backreaction}\label{intro:finalsec}
In Sections \ref{EinEulBR} - \ref{EinEulBR-3} we prove Theorem \ref{thm:intro2}. The main task is to apply the aforementioned transformation to the setting of a dynamical  spacetime geometry. This requires controlling various error terms based on suitable decay properties of the perturbations of the spacetime geometry. 

A key challenge arises from the non-trivial spatial curvature of the expansion-normalized spatial metric. In particular, unlike the flat scenario considered in \cite{FOW:2021} and Theorem \ref{thm:intro1}, the analysis leads to certain non-vanishing `problematic' terms that arise from the commutation of curved covariant derivatives. For example, the lowest order problematic term is
\begin{equation*}
	\int_{M}\left( g^{mn}\nabla_{m}z^{a}[\nabla_{n},\nabla_{a}]\psi+g^{mn}\nabla_{m}\psi[\nabla_{n},\nabla_{a}]z^{a} \right).
\end{equation*}
Up to a negligible error term, this reduces to 
\begin{equation}\label{eq:intro_baddie}
	\frac{2}{9}\int_{M}z^{a}\nabla_{a}\psi.
\end{equation}
Such a nonlinearity cannot be absorbed into a generic error term due to the structure of the energy estimates. Schematically, the energy estimate looks like
\begin{equation*}
	\partial_{T}\left( \| \psi \|_{H^1(M)} + \| z\|_{H^1(M)}\right)  \lesssim -C(K)\| z\|_{H^1(M)} + \eqref{eq:intro_baddie} + \text{error},
\end{equation*}
for some fixed $ C(K)>0 $ involving the equation of state parameter $K$. The expression \eqref{eq:intro_baddie} is, however, only linear in $ z $ and hence it cannot be absorbed into the negative definite term. 

To overcome this obstacle, we set up in Section \ref{corr-en} a correction mechanism for the original Fuchsian energy functionals for the fluid as given in Definition \ref{def:energy}. This technique is usually applied to geometric wave equations and consists of adding an indefinite term to a standard geometric $L^2$-energy, which leads to a cancellation of the problematic terms in the energy equality while preserving coercivity of the energy (see e.g. \cite{AnderssonMoncrief:2011}). 
For the first order, our method boils down to adding a correction
\begin{equation}\label{eq:intro_correction}
	 -\frac{1}{9}\int_{M}|z|^{2}.
\end{equation}
Note that this correction term is of order zero, i.e.~one order lower than the energy. We then need to take the time evolution of the correction term \eqref{eq:intro_correction} using the equations of motion \eqref{eq:intro_mod_E_EoM}.  Properties of the matrix $A^k$ and projection term $\Pbb$ (given precisely in \eqref{Milne-Ckeq}) yield a term of the form
\begin{equation*}
	-\frac{2}{9}\int_{M} \langle \Pbb Z, \Pbb C \nabla Z \rangle =-\frac{2}{9}\int_{M} z^{a}\nabla_{a}\psi
\end{equation*}
in the energy estimates, which exactly cancels the problematic term \eqref{eq:intro_baddie}. Note that in this case the coefficient in front of the correction is of small modulus and hence does not compromise the equivalence of the corrected energies to the original ones. For details on how to treat more complicated higher-order terms, see Proposition \ref{finalenergy}.

\subsection{Conclusions}
Given a spacetime of the form \eqref{background-cosmology} with $a(\tb)\approx \tb^\alpha$, we define, in this paper, the threshold rate $\alpha_c(K)$ to be the critical rate above which fluid stabilization occurs and below which fluid singularities form from arbitrary small perturbations of homogeneous solutions. \textit{A priori} the threshold rate also depends on the presence of gravitational backreaction and whether general fluids are considered or only the irrotational data.
From the perspective of cosmology the threshold rate sets limits on the epoch of cosmological evolution where structure formation could have occurred. 
 
The present paper shows that fluid stabilization occurs in the presence of gravitational backreaction in the regime of linear expansion for the full range
$K \in (0,1/3)$. Thus, together with \cite{FOfW:2021}, we obtain
\begin{equation*}
\alpha_c(K)<1 \qquad\mbox{for}\quad K \in [0, 1/3).
\end{equation*}
Since it is known in addition that  fluids with $0\leq K<\frac13$ stabilize when the geometry undergoes accelerated expansion, we conclude that
structure formation in a cosmological spacetime filled with non-radiation relativistic fluids requires an epoch of \textit{de}celerated expansion.  

Finally, we mention that for fluids with $0\leq K<\frac13$ the present paper completes the analysis for linear expansion rates  in the neighborhood of the canonical Friedman models. 

\subsection{Acknowledgements}
D.F. and M.O. acknowledge support by the Austrian Science Fund (FWF): [P 34313-N]. Z.W. thanks the African Institute of Mathematical Sciences Rwanda for hospitality during part of the work of this paper. 

\section{The Fluid Transformation}\label{sec:transf}
In this section we introduce the transformation of the fluid variables.
\subsection{The conformal Euler equations}
 Rather than working with the physical variables  $\gb_{\mu\nu}$ and $\vb^\mu$, we begin by defining conformal variables.  
 
\begin{Def}[Conformal variables $g, v^\mu, \Gamma, \nabla$]
For $\Psi$ an arbitrary scalar, define the \textit{conformal metric} $g_{\mu\nu}$ and the \textit{conformal four-velocity} $v^\mu$ by
\begin{equation} \label{g-v-def}
    \gb_{\mu\nu}= e^{-2\Psi}g_{\mu\nu}, \quad v^\mu = e^{-\Psi}\vb^\mu. 
\end{equation}
Let $\Gammab^\gamma_{\mu\nu}$ and $\Gamma^\gamma_{\mu\nu}$ denote the Christoffel symbols of $\gb_{\mu\nu}$ and $g_{\mu\nu}$, respectively, and let $\nabla_\mu$ denote the Levi-Civita connection of $g_{\mu\nu}$.
\end{Def}

\begin{lem}[The conformal Euler equations]
The Euler equations \eqref{rel-Eul} can be rewritten in terms of conformal variables as
\begin{align}
    v^\gamma\del{\gamma}\rhob + (\rhob+\pb) L^\gamma_j \delta^j_\nu \nabla_\gamma v^\nu &= 3(\rhob+\pb)v^\gamma \del{\gamma}\Psi, \label{conf-Eul-B.1} \\
    (\rhob+\pb)M_{ij}\delta^j_\nu v^\gamma\nabla_\gamma v^\nu + L^\gamma_i \del{\gamma}\pb &= (\rhob+\pb)L^\gamma_i \del{\gamma}\Psi, \label{conf-Eul-B.2}
\end{align}
where 
\begin{equation}\label{Lgammai-Mij-def}
    L^\gamma_i \define \delta^\gamma_i -\frac{v_i}{v_0}\delta^\gamma_0, \AND 
    M_{ij}\define g_{ij}-\frac{v_i}{v_0}g_{0j}-\frac{v_j}{v_0}g_{0i}+\frac{v_iv_j}{v_0^2}g_{00}.
\end{equation}
\end{lem}

\begin{proof}
 By \eqref{g-v-def} it is straightforward to verify that the Christoffel symbols are related via
\begin{equation*}
    \Gammab^\gamma_{\mu\nu} -\Gamma^\gamma_{\mu\nu} = - g^{\gamma\lambda}(g_{\mu\lambda}\del{\nu}\Psi + g_{\nu\lambda}\del{\mu}\Psi - g_{\mu\nu}\del{\lambda}\Psi).
\end{equation*}
With the help of this relation, we can express the Euler equations \eqref{rel-Eul} as
\begin{equation}\label{conf-Eul-A}
    \nabla_\mu \Tb^{\mu\nu} = 6 \Tb^{\mu\nu}\del{\mu}\Psi - g_{\lambda\gamma}\Tb^{\lambda\gamma}g^{\mu\nu}\del{\mu}\Psi.
\end{equation}
We refer to these equations as the \textit{conformal Euler equations}. 
In \cite{Oliynyk:CMP_2015}, see, in particular, equations (2.36)-(2.37) and (2.40)-(2.43) from that article, it is shown that, in an arbitrary coordinate system $(x^\mu)$, the conformal Euler equations \eqref{conf-Eul-A} reduce to \eqref{conf-Eul-B.1}-\eqref{conf-Eul-B.2}.
\end{proof}

\begin{Def}[The ADM decomposition of the conformal metric]
Let $(I\times \Sigma,g)$ be a Lorentzian spacetime where $I$ is an interval in $\Rbb$ and $\Sigma$
is a three manifold. We assume that $t=x^0$ is a time function, that is it takes values in $I$, and that the level sets $\Sigma_{\tau}=t^{-1}(\tau)$, $\tau\in I$, are diffeomorphic to $\Sigma$. If $n$ denotes the unit conormal to the spatial slices $\Sigma_t$, then we can express it as
\begin{equation} \label{n-down-def}
n=-\alpha dt \quad \Longleftrightarrow \quad n_\mu = -\alpha \delta_\mu^0,
\end{equation}
where the positive function $\alpha$ is known as the \textit{lapse} and in the following we view it as a time-dependent scalar field on $\Sigma$. 
The \textit{shift vector} $   \beta=\beta^i\del{i}$,
which we can view as a time-dependent vector field on $\Sigma$,
is then determined via the expression
\begin{equation*}
 n^\sharp =\frac{1}{\alpha}(\del{t}-\beta) \quad \Longleftrightarrow \quad  n^\mu = \frac{1}{\alpha}(\delta^\mu_0-\beta^i\delta_i^\mu),
\end{equation*}
that characterises the difference between the covariant form $n^\sharp$ of the normal vector and $\del{t}$.

The $3+1$ decomposition of the conformal spacetime metric $g$ is then given by
\begin{equation}\label{3+1-g}
    g= -\alpha^2 dt\otimes dt +\gc_{ij}(dx^i +\beta^i dt)\otimes (dx^j +\beta^j dt),
\end{equation}
where $ \gc=\gc_{ij}dx^i\otimes dx^j $
is the induced
%\footnote{In other words, if $\iota_t\: :\: \Sigma_t \rightarrow M$ is the inclusion map, then $\gc=\iota_t^*g$.}
 spatial metric on the spatial slices $\Sigma_t$, which we view as a time-dependent Riemmanian metric on $\Sigma$, while the shift and lapse are defined by $\beta_j=g_{j0}$ and $\alpha = (-g^{00})^{-\frac{1}{2}}$, respectively. We denote the Levi-Civita connection of $\gc_{ij}=g_{ij}$ by $\Dc_k$ and the Christoffel symbols $\gamma^k_{ij}$.
\end{Def}

Spacetime co-vector fields can be projected to three dimensional co-vector fields using the operator  \begin{equation}\label{P-def}
    P^\mu_i = \delta^\mu_i,
\end{equation}
which we note by \eqref{n-down-def} satisfies $P^\mu_i n_\mu = 0$. 
By our conventions, where the spacetime and spatial metrics are respectively used to raise and lower spacetime and spatial indices, we have from \eqref{P-def} that 
\begin{equation} \label{P-rl-def}
    P^i_\mu = g_{\mu\nu}\gc^{ij}P_j^\nu.
\end{equation}
An explicit formula for this operator is then easily computed from \eqref{3+1-g} and \eqref{P-def} and is given by
\begin{equation} \label{P-up-form}
P^i_\mu =     \delta^i_\mu + \beta^i\delta^0_\mu.
\end{equation}
It is also worth noting that the spacetime metric can be expressed using the operator \eqref{P-rl-def} as $g_{\mu\nu}=\gc_{ij}P^i_\mu P^j_\nu - n_\mu n_\nu$ and that the identities $P^\mu_i P_\mu^j = \delta^j_i$ and $P^i_\mu n^\mu =0$ hold.

\begin{Def}[Christoffel components $\Upsilon^i$ and $\Xi^i_j$]
A calculation shows that the Christoffel symbols $\Gamma^\gamma_{\mu\nu}$ of the four dimensional conformal metric $g_{\mu\nu}$ are given by the following:
\begin{equation}\label{3+1-Christ-0s}
\Gamma^{0}_{00} = \frac{1}{\alpha}(\del{t}\alpha+\beta^j \Dc_j \alpha - \beta^i \beta^j \Kc_{ij} ) , 
\qquad 
\Gamma^{0}_{i0} = \frac{1}{\alpha}(\Dc_i\alpha- \beta^j \Kc_{ij}) ,
\qquad
\Gamma^{0}_{ij} =-\frac{1}{\alpha}\Kc_{ij},
\end{equation}
and
\begin{equation}\label{3+1-Christ-ij0}
\begin{split}
\Gamma^{i}_{00} &=\Upsilon^i\define\alpha \Dc^i \alpha - 2\alpha\beta^j \Kc_j^i-\frac{1}{\alpha}(\del{t}\alpha+\beta^j\Dc_j\alpha-\beta^j\beta^k\Kc_{jk})\beta^i+\del{t}\beta^i +\beta^j\Dc_j\beta^i , \\
\Gamma^{i}_{j0} &=\Xi^i_j\define -\frac{1}{\alpha}\beta^i(\Dc_j \alpha -\beta^k\Kc_{kj})-\alpha \Kc_j^i +\Dc_j \beta^i ,
\end{split}
\end{equation}
and
\begin{equation}\label{3+1-Christ-kij}
\Gamma^{k}_{ij} = \gamma^k_{ij} +\frac{1}{\alpha}\beta^k\Kc_{ij} ,
\end{equation}
where $\Kc_{ij} = -\frac{1}{2\alpha}(\del{t}\gc_{ij} -\Dc_i \beta_j-\Dc_j\beta_i)$ and $\Kc_i^j = \gc^{ik}\Kc_{ik}$.
\end{Def}

\begin{rem}
 It is worth noting from  \eqref{3+1-Christ-0s}-\eqref{3+1-Christ-kij} that each of the groups of Christoffel symbols define a time-dependent tensor field on $\Sigma$ except for the last one \eqref{3+1-Christ-kij}, which is not a tensor due to the appearance of the Christoffel symbols $\gamma^k_{ij}$.
\end{rem}

\subsection{The ADM decomposition of the conformal Euler equations}
We first use the normal vector $n_\mu$ and the operator $P^\mu_i$ to decompose the four-velocity.
\begin{Def}[Decomposition of conformal four-velocity $\nu, w_j, \mu, u^j$]
We define
\begin{equation} \label{conform-decomp}
    \nu\define-\frac{1}{\alpha} n_\nu v^\nu , \quad w_j\define P^\mu_j v_\mu, \quad \mu %:= - \alpha^2 \nu + \beta^i w_i
    \define(\alpha n^\mu + \beta^i P_i^\mu)v_\mu \AND u^j \define P^j_\mu v^\mu -\nu\beta^j.
\end{equation}
\end{Def}
On account of \eqref{n-down-def}-\eqref{P-def} and \eqref{P-rl-def}-\eqref{P-up-form}, we have
\begin{equation*}
    \nu = v^0, \quad w_j=v_j, \quad \mu =v_0 \AND u^j = v^j.
\end{equation*}
Furthermore, using this notation, we observe from \eqref{3+1-g} and \eqref{conform-decomp} that $M_{ij}$ can be expressed as
\begin{equation}\label{Mij-3+1}
M_{ij}=\gc_{ij}-\frac{1}{\mu}(w_i\beta_j +w_j \beta_i) + \frac{-\alpha^2+|\beta|_{\gc}^2}{\mu^2}w_i w_j.
\end{equation}
We further observe that the fields \eqref{conform-decomp} are not independent due to the relation $v_\mu =g_{\mu\nu}v^\nu$ and, on account of \eqref{vb-norm} and \eqref{g-v-def},
the following normalization condition holds
\begin{equation} \label{v-norm}
    g_{\mu\nu}v^\mu v^\nu= -1.
\end{equation}
In fact, $\nu$, $w_j$ and $\mu$ can be expressed in terms of $u^j$. To see why this is the case, we observe from \eqref{conform-decomp} and \eqref{P-rl-def} that
\begin{equation*}
    u_j = P^\mu_j v_\mu - \nu\beta_j=w_j -\nu \beta_j
\end{equation*}
or equivalently
\begin{equation} \label{wj-form}
    w_j = u_j + \nu \beta_j.
\end{equation}
Using this, we then have by \eqref{conform-decomp} that
\begin{equation} \label{mu-form}
    \mu = \beta^i u_i + (|\beta|_{\gc}^2 -\alpha^2)\nu.
\end{equation}
Additionally, by \eqref{3+1-g}, \eqref{conform-decomp} and \eqref{v-norm}, we have
that
\begin{equation*}
    (-\alpha^2+|\beta|_{\gc}^2)\nu^2 + 2\beta_j u^j \nu + 1+|u|_{\gc}^2 =0.
\end{equation*}
Solving this quadratic equation for $\nu$, we find the two roots given by
\begin{equation} \label{nu-form}
\nu =\frac{1}{-\alpha^2+|\beta|^2_{\gc}}
\Bigl(-\beta_j u^j \pm \sqrt{(\beta_j u^j)^2 +(1+|u|_{\gc}^2)(\alpha^2-|\beta|^2_{\gc})}\Bigr).
\end{equation}

\begin{rem} 
 The choice of root, i.e. the choice of the $\pm$ sign, in \eqref{nu-form} determines the orientation of the conformal four velocity $v^\mu$. If we want $v^\mu$ to point in the direction of \textit{increasing} $t$ then we take the ``$-$'' sign. On the other hand, if we want $v^\mu$ to point in the direction of \textit{decreasing} $t$, then we take the ``$+$'' sign. In either case, we always have that
 \begin{equation*}
     \mu \nu < 0,
 \end{equation*}
 as can be verified from \eqref{mu-form} and \eqref{nu-form}.
 \end{rem}
 
\begin{lem}[The conformal Euler equations in ADM variables]
The conformal Euler equations \eqref{conf-Eul-B.1}-\eqref{conf-Eul-B.2} can be written as
\begin{align}
    \del{t}\rhob -\frac{(\rhob+\pb)}{\nu\mu}w_j\del{t}u^j+\frac{1}{\nu}u^k \Dc_k \rhob +\frac{(\rhob+\pb)}{\nu}\Dc_ju^j &= 3(\rhob+p)\Bigl(\del{t}\Psi + \frac{1}{\nu}u^k \Dc_k \Psi\Bigr) +  (\rhob+\pb)\ell   , \label{conf-Eul-C.1} \\
    (\rhob+p)M_{ij}\del{t}u^j-\frac{K}{\nu\mu}w_i \del{t}\rhob+ \frac{(\rhob+\pb)}{\nu}M_{ij}u^k\Dc_k u^j + \frac{K}{\nu} \Dc_i \rhob &= (\rhob+p)\Bigl(-\frac{1}{\nu\mu}w_i\del{t}\Psi + \frac{1}{\nu}\Dc_i\Psi\Bigr)+(\rhob+\pb)m_i, \label{conf-Eul-C.2}
\end{align}
where, from the barotropic equation of state  $ \pb = f(\rhob)$, the square of the sound speed is
\begin{equation} \label{K-def}
    K = \frac{dp}{d\rhob}=f'(\rhob)
\end{equation}
and we define
\begin{align}
 \ell &\define -\frac{1}{\nu\alpha}\Kc_{jk}\beta^ju^k-\Xi^j_j +\frac{1}{\mu}\Upsilon^j w_j + \frac{1}{\nu\mu}\Xi^j_k w_j u^k \label{ell-def}
 \intertext{and}
 m_i &\define - M_{ij}\Bigl( \frac{1}{\nu\alpha}\Kc_{lk}\beta^j u^l u^k + \nu \Upsilon^j +2 \Xi^j_k u^k\Bigr). \label{mi-def}
\end{align}
\end{lem}

\begin{proof}
For an arbitrary scalar field $h$, we observe from \eqref{Lgammai-Mij-def} and \eqref{conform-decomp} that
\begin{equation} \label{L-del-h}
L^\gamma_i \del{\gamma}h =-\frac{1}{\mu}w_i\del{t}h+\Dc_i u.
\end{equation} 
We also observe using \eqref{Lgammai-Mij-def} and \eqref{conform-decomp} that
\begin{align*} %\label{L-nabla-nu-A}
    L^\gamma_j \delta^j_\nu \nabla_\gamma v^\nu &= L^\gamma_j \delta^j_\nu(\del{\gamma}v^\nu + \Gamma^\nu_{\gamma\sigma}v^\sigma) \\
    &= -\frac{1}{\mu}w_j\del{t}u^j +\del{j}u^j + \Gamma^j_{j0}\nu + \Gamma^j_{jk}u^k -\frac{1}{\mu}w_j \Gamma^j_{00}\nu - \frac{1}{\mu}w_j \Gamma_{0k}^j u^k.
\end{align*}
Then, with the help of \eqref{3+1-Christ-0s}, \eqref{3+1-Christ-ij0} and \eqref{3+1-Christ-kij}, it follows that we can write the above expression as 
\begin{equation}\label{L-nabla-nu-A}
    L^\gamma_j \delta^j_\nu \nabla_\gamma v^\nu= -\frac{1}{\mu}w_j\del{t}u^j +\Dc_ju^j  +\frac{1}{\alpha}\Kc_{jk}\beta^ju^k+\Xi^j_j\nu -\frac{\nu}{\mu}\Upsilon^j w_j - \frac{1}{\mu}\Xi^j_k w_j u^k.
\end{equation}
Using a similar calculation, it is also not difficult to verify that
\begin{equation}\label{L-nabla-nu-B}
\delta^j_\nu v^\gamma \nabla_\gamma v^\nu = 
\nu \del{t}u^j + u^k\Dc_k u^j + \frac{1}{\alpha}\Kc_{ik}\beta^j u^i u^k + \nu^2 \Upsilon^j +2\nu \Xi^j_i u^i.
\end{equation}
The identities \eqref{L-del-h}, \eqref{L-nabla-nu-A} and \eqref{L-nabla-nu-B} then allows us to obtain the required form of the conformal Euler equations \eqref{conf-Eul-B.1}-\eqref{conf-Eul-B.2}.
\end{proof}

To proceed, we introduce a modified density variable $\zeta$ defined by subtracting $3\Psi$ from the fluid enthalpy.
\begin{Def}[Modified fluid density variable $\zeta$]
For a given barotropic equation of state  $ \pb = f(\rhob)$, the \textit{modified fluid density} is defined by
\begin{equation} \label{zetadef}
    \zeta \define \int^{\rhob}_{\rhob_0}\frac{d\xi}{\xi+f(\xi)} - 3\Psi,
\end{equation}
where $\rhob_0$ is any positive constant.
\end{Def}

\begin{rem}
For the linear equation of state \eqref{eos-lin}, we note that square of the sound speed $K$ is constant and lies in $[0,1]$, by assumption, while the 
the modified fluid density is given by
\begin{equation*} 
\zeta = \int^{\rhob}_{\rhob_0}\frac{d\xi}{(1+K)\xi} - 3\Psi= \frac{1}{1+K}\ln\Bigl(\frac{\rhob}{\rhob_0}\Bigr)-3 \Psi.
\end{equation*}
\end{rem}

\begin{lem}\label{lem:conf-Eul-U}
The conformal Euler equations \eqref{conf-Eul-C.1}-\eqref{conf-Eul-C.2} can be written as
\begin{equation} \label{conf-Eul-E}
B^0 \del{t}U + B^k\Dc_k U = H,
\end{equation}
where $U= (\zeta,u^j)^{\tr}$, and the matrices $B^0, B^k, H$ are given by 
\begin{align}
B^0 &= \begin{pmatrix}K & -\frac{K}{\nu\mu}w_j \\
-\frac{K}{\nu\mu}w_i & M_{ij} \end{pmatrix}, \quad
B^k = \begin{pmatrix}  \frac{K}{\nu}u^k & \frac{K}{\nu}\delta^k_j \\
\frac{K}{\nu} \delta^k_i & \frac{1}{\nu} M_{ij} u^k \end{pmatrix}, \AND \begin{pmatrix} K \ell \\
\Bigl(\frac{3 K -1}{\nu\mu}w_i\del{t}\Psi + \frac{1- 3K}{\nu}\Dc_i\Psi\Bigr)+m_i\end{pmatrix}.\label{B-def}
\end{align}
\end{lem}

\begin{proof}
 Differentiating the expression \eqref{zetadef}, we find that
\begin{equation*}
    \del{t}\zeta = \frac{\del{t}\rhob}{\rhob+\pb} -3 \del{t}\Psi
\AND
    \Dc_k\zeta =  \frac{\Dc_k\rhob}{\rhob+\pb} -3 \Dc_k\Psi.
\end{equation*}
With the help of these expressions, a short calculation shows that the conformal Euler equations \eqref{conf-Eul-C.1}-\eqref{conf-Eul-C.2} when expressed in terms of the modified density $\zeta$ become
\begin{align*}
    \del{t}\zeta -\frac{1}{\nu\mu}w_j\del{t}u^j+\frac{1}{\nu}u^k\Dc_k \zeta +\frac{1}{\nu}\Dc_ju^j &= \ell   , \\
    M_{ij}\del{t}u^j-\frac{K}{\nu\mu}w_i \del{t}\zeta+ \frac{1}{\nu}M_{ij}u^k\Dc_k u^j + \frac{K}{\nu}\Dc_i \zeta &= \Bigl(\frac{3 K -1}{\nu\mu}w_i\del{t}\Psi + \frac{1- 3K}{\nu}\Dc_i\Psi\Bigr)+m_i. 
\end{align*}
It is then a simple exercise to put these  into the desired matrix form. 
\end{proof}

\begin{rem}
It is worth noting that \eqref{conf-Eul-E} is manifestly symmetric hyperbolic, and consequently represents a symmetric hyperbolic formulation of the conformal Euler equations. The symmetric hyperbolic nature of this formulation guarantees the existence of local-in-time solutions by standard existence results for symmetric hyperbolic systems of equations, e.g. see Propositions 1.4, 1.5 and 2.1  from \cite[Chapter 16]{TaylorIII:1996}. While this at least yields the existence of local solutions, there is still the question of whether or not global solutions exist. 
\end{rem}

In order to address the long-time existence of solutions, we need to bring the system \eqref{conf-Eul-E} into a more favourable form. This will be accomplished by the following change of variables.

\begin{Def}[Fluid transformation]
Define $Z= (\psi,z^j)^{\tr}$ and a transformation given by
\begin{equation}\label{cov}
  U=(\zeta,u^i)^{\tr} =:\bigl( a(\psi,|z|^2_{\gc}), b(\psi)z^i\bigr)^{\tr},
\end{equation}
where $a(\cdot,\cdot)$ and $b(\cdot)$ are, for the moment, arbitrary functions.
\end{Def} 

\begin{lem}\label{lem:transf-Eul-general}
If we choose the functions $a(\psi,|z|^2_{\gc})$ and $b(\psi)$ as
\begin{align*} 
    a(\psi, |z|^2_{\gc}) &= c_1- \ln(4)+\ln((\psi+c_2)^2)+ \frac{\kappa |z|_{\gc}^2}{(\psi+c_2)^2}, \quad \text{ and } \quad
    b(\psi) = \frac{2}{\psi+c_2},
\end{align*}
where $ c_{1,2} $ are arbitrary constants and $ \kappa\define K^{-1}-2 $, then the equations given in \eqref{conf-Eul-E} take form
\begin{equation}\label{conf-Eul-F}
A^0\del{t}Z + \frac{1}{\nu}A^k \Dc_k Z = Q^{\tr}(H-B^0Y),
\end{equation}
and the following conditions hold:
\begin{align}
    A^0_j &= \Ord\bigl(|\beta|_{\gc} +|z|_{\gc}^2\bigr), \qquad
    A^k_0 =  \Ord\bigl(|\beta|_{\gc} +|z|_{\gc}^2\bigr) ,\label{A-exp-A}
    \intertext{and}
    \frac{D_1 a}{b} &= c_0 +  \Ord\bigl(|\beta|_{\gc} +|z|_{\gc}^2\bigr), \label{ab-fix-A}
\end{align}
for $c_0$ a non-zero constant, where
\begin{equation*}
    A^0 =\begin{pmatrix} A^0_0 & A^0_j \\
    A^0_i & A^0_{ij}
    \end{pmatrix} \AND A^k =\begin{pmatrix} A^k_0 & A^k_j \\
    A^k_i & A^k_{ij}
    \end{pmatrix}.
\end{equation*}
\end{lem}
\begin{rem}
    Note that the explicit form of the functions $ a $ and $ b $ given in Lemma \ref{lem:transf-Eul-general} are smooth in a neighborhood of $ (\psi,|z|_{\gc}^{2})=(0,0) $ for the choice of $ c_{2}\neq 0 $. 
\end{rem}
\begin{proof}
Differentiating $U$, we find after a short calculation that
\begin{align} \label{dU-2-dZ}
    \del{t}U = Q\del{t}Z + Y \AND \Dc_{k}U = Q\Dc_k U,
\end{align}
where 
\begin{align}
    Q&= \begin{pmatrix} D_1 a & 2 D_2 a z_j \\
    b'z^i & b\delta^i_j\end{pmatrix} \label{Q-Y-def}
    \quad \text{ and } \quad
    Y= \begin{pmatrix} D_2 a \del{t}\gc_{ij} z^i z^j \\ 0 \end{pmatrix} .
\end{align}
Here $D_1 a$ and $D_2 a$ denote the partial derivatives with respect to the first and second variables of $a=a(\psi,|z|^2_{\gc})$. 
Using \eqref{dU-2-dZ} and the fact that $\Dc_k \gc_{ij}=0$, it is then clear that we can write \eqref{conf-Eul-E} as
\begin{equation*}
A^0\del{t}Z + \frac{1}{\nu}A^k \Dc_k Z = Q^{\tr}(H-B^0Y),
\end{equation*}
where
\begin{equation*} 
    A^0=Q^{\tr}B^0 Q, \quad A^k=Q^{\tr}A^k Q,  \AND \label{Q-tr}
     Q^{\tr}= \begin{pmatrix} D_1 a &b'z^j  \\
    2 D_2 a z_i & b\delta^i_j\end{pmatrix}.
\end{equation*}

We then have by \eqref{B-def}, \eqref{Q-Y-def} and \eqref{Q-tr}, that
\begin{align}
    A^0_0 &= K (D_1 a)^2 - \frac{2 K b' D_1 a }{\nu\mu}z^k w_k + (b')^2M_{kl}z^k z^l, \label{A00} \\
    A^0_j &= 2K D_1 a D_2 a z_j -2\frac{K b'}{\nu\mu}D_2 a z_j z^k w_k -\frac{K b D_1 a}{\nu\mu} w_j + bb' z^k M_{kj}, \notag\\
    A^0_{ij} &= b^2 M_{ij} - \frac{2 K b D_2a }{\nu \mu}(z_i w_j + z_j w_i) + 4K (D_2 a)^2 z_i z_j ,\notag\\
      A^k_0 &= K (D_1 a)^2u^k+2 K b' D_1 a z^k+(b')^2M_{lm}z^l z^m u^k,\label{Ak0} \\
    A^k_j &=K b D_1 a\delta^k_j+ 2 K D_1 a D_2 a u^k z_j + 2 K b' D_2 a z^k z_j + b b'M_{lj}z^l u^k, \notag
    \intertext{and}
    A^k_{ij} &=b^2 M_{ij} u^k + 2 K b D_2 a (\delta^k_j z_i +\delta^k_i z_j) + 4 K (D_2 a)^2u^k z_i z_j.\label{Akij}
\end{align}

Now, our goal is to try and choose $a$ and $b$ so that \eqref{A-exp-A}-\eqref{ab-fix-A} hold.
By \eqref{nu-form} and \eqref{cov}, we observe that
\begin{equation*} 
    \nu = -\frac{1}{\alpha}+ \Ord\bigl(|\beta|_{\gc} +|z|_{\gc}^2\bigr),
\end{equation*}
which, in turn, implies by \eqref{mu-form} that
\begin{equation*} 
\mu = \alpha + \Ord\bigl(|\beta|_{\gc} +|z|_{\gc}^2\bigr).
\end{equation*}
Using these, we then see from \eqref{Mij-3+1}, \eqref{wj-form} and \eqref{cov} that
\begin{equation}\label{A0j-exp-B}
    A^0_j = \bigl(KD_1 a(2 D_2 a + b^2)+b b'\bigr)z_j+\Ord\bigl(|\beta|_{\gc} +|z|_{\gc}^2\bigr).
\end{equation}
We see also from \eqref{cov} and \eqref{Ak0} that 
\begin{equation} \label{Ak0-exp-B}
    A^k_0 = K D_1 a(b D_1 a + 2 b')z^k + \Ord(|z|_{\gc}^2).
\end{equation}

By \eqref{A0j-exp-B} and \eqref{Ak0-exp-B}, we then see that if $a$ and $b$ are chosen so that
\begin{align}
    b D_1 a + 2b' &= \Ord(|z|_{\gc}^2), \label{ab-fix-B.1} \\
    KD_1 a(2 D_2 a + b^2)+b b'&= \Ord(|z|_{\gc}^2),\label{ab-fix-B.2} 
\end{align}
then the conditions \eqref{A-exp-A}-\eqref{ab-fix-A} are satisfied. Now, writing \eqref{ab-fix-B.1} as
\begin{equation*}
    \del{\psi}\bigl( a(\psi,|z|_{\gc}^2)+ \ln\bigl( b(\psi)^2\bigr) \bigr) = \Ord(|z|_{\gc}^2),
\end{equation*}
it is clear that we can ensure that this condition holds by choosing $a$ so that
\begin{equation} \label{ab-fix-C}
    a(\psi,|z|_{\gc}^2) = c_1 -\ln(b(\psi)^2)+c(\psi)|z|_{\gc}^2,
\end{equation}
where $c(\psi)$ is an arbitrary function and $c_1$ is an arbitrary constant. Inserting this into \eqref{ab-fix-B.2} and multiplying through by $\frac{b}{b'}$ yields
\begin{equation*}
-2 K(2c+b^2) + b^2 = \Ord(|z|_{\gc}^2),
\end{equation*}
where here we are viewing $K$ as a function of $\psi$ via\footnote{It will also, of course be a function of $\Psi$, but we can for the purpose of this argument treat it as a ``constant''.} \eqref{K-def}, \eqref{zetadef} and \eqref{cov}. It is then clear that we can guarantee that this condition holds by setting
\begin{equation*} 
    c= \frac{1}{4}\kappa b^2, \quad \kappa= K^{-1} -2.
\end{equation*}
Now, \eqref{ab-fix-C} implies that
\begin{equation} \label{ab-fix-D}
    a =c_1 - \ln(b^2)+ \frac{1}{4}\kappa b^2 |z|_{\gc}^2.
\end{equation}
Substituting this into \eqref{ab-fix-A}, it follows that \eqref{ab-fix-A} will hold provided 
\begin{equation*}
    -2\frac{b'}{b^2}= c_0.
\end{equation*}
Solving this yields 
\begin{equation*} 
b = \frac{2}{c_0 \psi + c_2},
\end{equation*}
where $c_2$ is an arbitrary constant. We then fix the constant $c_0$ by setting $ c_0=1$. 
By \eqref{ab-fix-D}, this gives 
\begin{align*} 
    a(\psi, |z|^2_{\gc}) &= c_1- \ln(4)+\ln((\psi+c_2)^2)+ \frac{\kappa |z|_{\gc}^2}{(\psi+c_2)^2}, \quad \text{ and } \quad
    b(\psi) = \frac{2}{\psi+c_2}. 
\end{align*}
By construction, the conditions \eqref{A-exp-A}-\eqref{ab-fix-A} with $c_0=1$ hold.
Note that this choice gives
\begin{equation}\label{transf-derivs1}
D_1 a = \frac{2}{\psi+c_2}\Bigl( 1 - \kappa \frac{|z|_{\gc}^2}{(\psi+c_2)^2}\Bigr), \quad D_2 a = \frac{\kappa}{(\psi+c_2)^2},
\quad b'(\psi) = \frac{-2}{(\psi+c_2)^2}.
\end{equation}
\end{proof}

\section{The Euler equations on  Milne-like FLRW spacetimes}\label{section:fixed_MFLRW}
We now apply the transformation of Section \ref{sec:transf} to the FLRW-type geometries considered in \cite{FOW:2021} with  a linear equation of state \eqref{eos-lin}.

\begin{Def}[Milne-like FLRW spacetimes]
Consider the manifold $\Rbb_{>0} \times \Tbb^3$
with the spacetime metric 
\begin{equation} \label{MFLRW-g}
    \gb = \frac{1}{t^2} g, \quad\text{where} \quad g= -\frac{1}{t^2}dt\otimes dt + \delta_{ij}dx^i \otimes dx^j.
\end{equation}
We refer to the pair $(\Rbb_{>0} \times \Tbb^3,\gb)$ 
as a \textit{Milne-like FLRW spacetime}. 
\end{Def}

Note that, with respect to the coordinates used in \eqref{MFLRW-g}, the future is located in the direction of decreasing $t$ and future timelike infinity is located at $t=0$. If we choose
\begin{equation*} 
    \Psi=\ln(t),
\end{equation*}
then the Milne-like FLRW metric $ \gb $ is of the form \eqref{g-v-def} where the confomal metric is given by $ g $ as written in \eqref{MFLRW-g}. Moreover, the $3+1$ decomposition of the conformal metric $g$ in \eqref{MFLRW-g} yields 
\begin{align}\label{MFLRW}
\alpha = t^{-1},\quad \beta_j = 0, \quad
\gc_{ij}= \delta_{ij}.
\end{align}
The spatial manifold on which these fields are defined is $\Tbb^3$.

\begin{lem}\label{MFLRW-transformation-variables}
On the Milne-like FLRW spacetime we have the simplifications 
\begin{equation}\label{MFLRW-vars}\begin{split}
\nu &= v^0, \; w_j = v_j, \;\mu = - t^{-2} \nu,\; u^j = v^j, \\
M_{ij} &= \delta_{ij} - \Bigl( \frac{t}{\nu}\Bigr)^2 v_i v_j = \delta_{ij} - \Bigl( \frac{t}{\nu}\Bigr)^2 b^2 z_i z_j, \AND 
H= \begin{pmatrix} 0 \\ \frac{3K-1}{t \nu \mu} bz_j\end{pmatrix},
\end{split}\end{equation}
where we defined $z_i = \delta_{ij} z^j$. 
Moreover, $\mathcal{K}_{ij}, \Upsilon^i, \Xi^i_j, m_i, \ell, Y$ all identically vanish. 
\end{lem}
\begin{proof}
This is a straightforward computation using \eqref{MFLRW} in the appropriate definitions, and lowering and raising indices with \eqref{MFLRW}.
\end{proof}

\begin{lem}[Homogeneous solutions]\label{homsolutions}
The homogeneous background solutions of the Euler equations \eqref{rel-Eul} on a fixed Milne-like FLRW spacetime $(\Rbb_{>0} \times \Tbb^3,\gb)$ are given by $U \equiv 0$.
Furthermore, this reduces to $Z \equiv 0$ and on these background solutions we have the simplifications
\begin{equation}\label{MFLRW-simp}
D_1 a  = (2/c_2), \quad D_2 a = \kappa/{c_2^2}, \quad b'= - 2/{c_2^2}, \quad 
\nu \mu = -1, \AND t/\nu = -1\,.
\end{equation} 
\end{lem}

\begin{proof}
The homogeneous solutions considered in \cite{FOW:2021} are given by 
\begin{equation*}
 (\vb^\mu_H, \rhob_H) = (-t^2 \delta^\mu_0, \big( (1-3K) c_H \big)^{\frac{1+K}{K}} t^{3(1+K)}),
 \end{equation*}
 where $c_H>0$ is a constant. On such a homogeneous solution we have 
$$v^\mu_H = -t \delta^\mu_0, \quad \zeta_H = 0\,,$$ 
provided we pick $ \rho_0 = ((1-3K)c_H)^{\frac{1+K}{K}}$. This implies that  $U\equiv 0$ and $Z= (C_H, 0)$ for some constant $C_H$. We can in fact set $C_H=0$ by picking $c_2 = - 2\exp( - c_1/2)$. 

Inspecting the statements given in \eqref{transf-derivs1} and simply substituting in $ Z\equiv 0 $ yields the desired result for $ D_{1}a $, $ D_{2}a $ and $ b'$ in \eqref{MFLRW-simp}. By Lemma \ref{MFLRW-transformation-variables} and using $ v^{0}=-t $ on the homogeneous solutions, we have that $ t/\nu=t/{v^{0}}=-1 $ as well as
\begin{equation*}
	\nu \mu = -t^{-2}\nu^{2}=-t^{-2}(v^{0})^{2}=-1.
\end{equation*}
\end{proof}

\begin{lem} The Euler equations on a fixed Milne-like FLRW spacetime take the form
\begin{equation}\label{relEuler_fixedMLRW}
M_0 \del{t} Z + \frac1t (\mathcal C^k + M^k(Z)) \del{k}Z = \frac1t \Bc(Z) \Pi Z + \frac1t F(Z)\,,
\end{equation}
where
\begin{equation}\begin{split}    \label{MLRW:Mk-ests}
M_0(Z) &
= \begin{pmatrix}
1 & 0 \\ 0 & K^{-1} \delta_{ij}
\end{pmatrix} + \begin{pmatrix}
0 &  \Ord\bigl(|z|_{\gc}^2\bigr) \\  \Ord\bigl(|z|_{\gc}^2\bigr) & \Ord\bigl(|z|_{\gc}^2\bigr)
\end{pmatrix}, \\
M^k(Z) &
= \frac{t}{\nu}  \begin{pmatrix} 0 & 0 \\
0 & K^{-1}\Bigl(\frac{2}{\psi+c_2}\Bigr)\delta_{ij} z^k + \frac{\kappa}{(\psi+c_2)} (\delta^k_j z_i +\delta^k_i z_j)
    \end{pmatrix}+ \Ord\bigl(|z|_{\gc}^2\bigr)\,,
     \\
 \mathcal C^k &= \begin{pmatrix} 0 & \delta^k_j \\
\delta^k_i & 0
    \end{pmatrix}\,, 
    \quad F(Z) = \Ord\bigl(|z|_{\gc}^2\bigr) \,,
\end{split}\end{equation}
and
\begin{align*}
 \Bc &:= (K^{-1}-3)\id, \qquad
\Pi := \begin{pmatrix} 0 & 0  \\0 & \delta^i_j\end{pmatrix}, 
\qquad \Pi^\perp := \id - \Pi = \begin{pmatrix} 1 & 0  \\0 & 0\end{pmatrix}.
\end{align*}
\end{lem}
\begin{rem}\label{M0component}
	Note that the non-leading order term of the upper left component of $ M^{0} $ is zero  and hence $ D\Pbb M^{0}\Pbb $ for any derivative $ D $.
\end{rem}
\begin{proof}
From the PDE system \eqref{conf-Eul-F}, we define
\begin{align*}
M_0(Z) &:= (A^0_0)^{-1} \begin{pmatrix} A^0_0 & A^0_j \\    A^0_i & A^0_{ij} \end{pmatrix},\quad \mathcal C^k := (A_0^0)^{-1}  \begin{pmatrix} A^k_0 & A^k_j \\    A^k_i & A^k_{ij} \end{pmatrix}\Big|_{z^i\equiv 0},
\quad
M^k(Z):=  \frac{t}{\nu}  (A^0_0)^{-1} \begin{pmatrix} A^k_0 & A^k_j \\    A^k_i & A^k_{ij} \end{pmatrix}- \mathcal C^k.
\end{align*}
Using \eqref{A00}-\eqref{Akij}, we compute
\begin{align*}
    A^0_0 &= \Bigl(\frac{2\sqrt{K}}{\psi+c_2}\Bigr)^2 + \Ord\bigl(|z|_{\gc}^2\bigr), \quad
    A^0_j =\Ord\bigl(|z|_{\gc}^2\bigr), \quad
    A^0_{ij} = \Bigl(\frac{2}{\psi+c_2}\Bigr)^2 \delta_{ij} +\Ord\bigl(|z|_{\gc}^2\bigr) ,
\\
      A^k_0 &= \Ord\bigl(|z|_{\gc}^2\bigr),\quad
    A^k_j =\Bigl(\frac{2\sqrt{K}}{\psi+c_2}\Bigr)^2 \delta^k_j + \Ord\bigl(|z|_{\gc}^2\bigr),\AND \\
    A^k_{ij} &=\Bigl(\frac{2}{\psi+c_2}\Bigr)^3\delta_{ij} z^k + \Bigl(\frac{2\sqrt{K}}{\psi+c_2}\Bigr)^2 \frac{\kappa}{(\psi+c_2)} (\delta^k_j z_i +\delta^k_i z_j) + \Ord\bigl(|z|_{\gc}^2\bigr).
\end{align*}
Multiplying the system \eqref{conf-Eul-F} from the left by  $(A_0^0)^{-1}$ yields
$$
(A_0^0)^{-1} A^0\del{t}Z + \frac{1}{t} \frac{t}{\nu} (A_0^0)^{-1}A^k \del{k} Z =(A_0^0)^{-1} Q^{\tr}(H-B^0Y) .
$$
Using that $(A_0^0)^{-1} =  \bigl(\frac{2\sqrt{K}}{\psi+c_2}\bigr)^{-2} + \Ord\bigl(|z|_{\gc}^2\bigr) $, we obtain
\begin{equation}\label{MLRW:M0}
\begin{split}
M_0(Z) &
= \begin{pmatrix}
1 & 0 \\ 0 & K^{-1} \delta_{ij}
\end{pmatrix} + \begin{pmatrix}
0 & \Ord\bigl(|z|_{\gc}^2\bigr)  \\ \Ord\bigl(|z|_{\gc}^2\bigr)  & \Ord\bigl(|z|_{\gc}^2\bigr) 
\end{pmatrix}
, \quad \mathcal C^k = \begin{pmatrix} 0 & \delta^k_j \\
\delta^k_i & 0
    \end{pmatrix}\,,\\
M^k(Z) &
= \frac{t}{\nu}  \begin{pmatrix} 0 & 0 \\
0 & \frac1{K}\Bigl(\frac{2}{\psi+c_2}\Bigr)\delta_{ij} z^k + \frac{\kappa}{(\psi+c_2)} (\delta^k_j z_i +\delta^k_i z_j)
    \end{pmatrix}+ \Ord\bigl(|z|_{\gc}^2\bigr) .
    \end{split}
\end{equation}
Finally, for the right hand side of the equation we use \eqref{Q-tr}, \eqref{MFLRW-vars} and \eqref{MFLRW-simp} to compute
$$ (A_0^0)^{-1} Q^{\tr}(H-B^0Y) = - \frac{1}{t}(1-3K)(\nu \mu)^{-1} b (A_0^0)^{-1} \begin{pmatrix} b' z^jz_j \\  b z_i\end{pmatrix}
= \frac1t(1-3K)K^{-1} \begin{pmatrix} 0 \\  z_i\end{pmatrix}+ \frac1t F(Z),
$$
where
\begin{equation}\label{MLRW-estsF}
F(Z) := t (A_0^0)^{-1} Q^{\tr}H- \Bc(Z)\cdot \Pi Z = - \frac{ (K^{-1}-3)}{(\psi+c_2)} \begin{pmatrix} z^j z_j \\  0\end{pmatrix}+ \Ord\bigl(|z|_{\gc}^2\bigr)  = \Ord\bigl(|z|_{\gc}^2\bigr). 
\end{equation}
\end{proof}

\begin{Def}[Matrix norm]
 For any $ M\in \mathbb{R}^{d\times d} $ we write
	\begin{equation*}
		|M|_{\op}=\sup\{|Mv|:v\in \mathbb{R}^{d}\}.
	\end{equation*}
		We will use this to estimate inner products by operator norms.
\end{Def}

\begin{rem}\label{rem:conjeffect}
The projection matrix $\Pbb$ is used to extract particular components of an arbitrary matrix 
$$ A = \begin{pmatrix}
A_0 & A_i \\ A_j & A_{ij}
\end{pmatrix}.
$$
For example,
$$ \Pbb M \Pbb = \begin{pmatrix}
0 & 0 \\ 0 & A_{ij} \end{pmatrix}$$
and so an estimate on $\Pbb A \Pbb$ is precisely an estimate on the lower-diagonal piece $A_{ij}$ of the matrix $A$. Similarly, $\Pbb A \Pbb^\perp$ extracts the $A_i$ component, $\Pbb^\perp A \Pbb$ the $A_j$ component and $\Pbb^\perp A \Pbb^\perp$ the $A_0$ component. 
With this perspective, we can also reinterpret condition \eqref{A-exp-A} as ensuring that the fluid PDE
\begin{equation}\label{eq:444}
A^0\del{t}Z + \frac{1}{\nu}A^k \Dc_k Z = Q^{\tr}(H-B^0Y) 
\end{equation}
enjoys the estimates
$$ |\Pbb A^0 \Pbb^\perp|_{\op} = |\Pbb^\perp A^0 \Pbb|_{\op} =  \Ord\bigl(|\beta|_{\gc} +|z|_{\gc}^2\bigr) \AND
   |\Pbb^\perp A^k \Pbb^\perp |_{\op} =  \Ord\bigl(|\beta|_{\gc} +|z|_{\gc}^2\bigr).
$$
Furthermore, the right hand side of \eqref{eq:444} explicitly reads
    \begin{equation*}
        \begin{pmatrix} D_1 a & b'z^i  \\
            2 D_2 a z_j& b\delta^i_j\end{pmatrix} \left(\begin{pmatrix} K \ell \\
\Bigl(\frac{3 K -1}{\nu\mu}w_i\del{t}\Psi + \frac{1- 3K}{\nu}\Dc_i\Psi\Bigr)+m_i\end{pmatrix}-\begin{pmatrix}K & -\frac{K}{\nu\mu}w_j \\
-\frac{K}{\nu\mu}w_i & M_{ij} \end{pmatrix}
        \begin{pmatrix} D_2 a \del{t}\gc_{ij} z^i z^j \\ 0 \end{pmatrix}\right).
    \end{equation*}
    If we assume perturbations of the metric to be negligible, we may set $ \ell=m_{i}=\partial_{t}\gc_{ij}=0 $. Furthermore, we may restrict ourselves to the case in which $ \Psi=\Psi(t) $ so that $\Dc \Psi=0$. In this setting we  have 
    \begin{equation*}
        Q^{\tr}(H-B^{0}Y)\simeq \frac{3 K -1}{\nu\mu}\begin{pmatrix}b^{\prime}z^{i}w_i\del{t}\Psi\\ bw_j\del{t}\Psi\end{pmatrix}.
    \end{equation*}
    As can be seen in the next section, an expanding FLRW--type model forces $ \nu\mu<0 $ and $ w\simeq b z $. Hence, we have that 
    \begin{equation*}
        Q^{\tr}(H-B^{0}Y)\simeq c(K)\begin{pmatrix} 0 \\ z^{j}\end{pmatrix} + \Ord(|z|^2) = c(K)\Pbb Z + \Ord(|z|^2),
    \end{equation*}
    for $ \partial_{t}\Psi>0 $ and some fixed constant $c(K)>0$. Note that this is in the regime of \emph{compactified} time. A switch to physical time translates $ c(K)\mapsto -c(K) $ and thus yields a \emph{negative} term. 
\end{rem}

We conclude this section with some estimates on the quantities appearing in \eqref{relEuler_fixedMLRW}.

\begin{lem}
	Let $ \delta>0 $ and assume that $ |Z|\leq \delta $. Then the following estimates hold:
	\begin{equation}\label{MLRW-ests1}
		\begin{split}
			|M_0(Z)|_{\op} + |\mathcal C^k|_{\op} + |(M_{0})^{-1}(Z)|_{\op} 
			&\lesssim 1,
			\\
			|M^k(Z)|_{\op} 
			&\lesssim |\Pi Z|,
		\end{split}
	\end{equation}
		and
	\begin{equation}\label{MLRW-ests2}
		\begin{split}
			|\Pbb\del{a}M^k(Z)\Pbb|_{\op} 
			&\lesssim |\Pi Z|+|\Pi DZ|,\\
			|\Pbb^{\perp}\del{a}M^k(Z)\Pbb|_{\op} +|\Pbb^{\perp}\del{a}M^k(Z)\Pbb^{\perp}|_{\op} +\quad &\\
			|\Pbb \del{a}M^k(Z)\Pbb^{\perp}|_{\op} 			+|\del{a}(M_0)(Z)|_{\op} 
			&\lesssim |\Pi Z|^{2}+|\Pi DZ|^{2},\\
		\end{split}
	\end{equation}
	and
	\begin{equation}\label{MLRW-estsFH}
		|F(Z)| \lesssim  |\Pi Z|^2 .
	\end{equation}
\end{lem}

\begin{rem}
	In this paper, see also the earlier work \cite{Oliynyk16}, we write $ |f|\lesssim |g| $, or $f = \Ord(g)$, to denote a universal constant $C>0$ such that $|f|\leq C |g|$ with $f, g$ some functions. Note that this is not simply an estimate of absolute values but also tacitly implies that
	\begin{equation*}
		f=h(g)
	\end{equation*}
	for some function $ h\in\mathcal{C}^{\infty} $ with $ a<|h|<b$, $a,b>0 $. Thus, for $ s\in \mathbb{N} $,
	\begin{equation*}
		\sum_{|l|\leq s }|\nabla^{l}f|\lesssim \sum_{|l|\leq s }|\nabla^{l}g|.
	\end{equation*}
\end{rem}

\begin{proof}
	It is straightforward to check \eqref{MLRW-ests1} using the explicit form of these matrices given in  \eqref{MLRW:Mk-ests} and \eqref{MLRW:M0}.
		Turning to \eqref{MLRW-ests2}, we  note that the leading order piece of $M^k$ is $ \Pbb M^{k}\Pbb $. Thus, we start with $|\Pbb \del{a}M^k(Z)\Pbb |_{\op} $. Using the expression for $\nu$ given in \eqref{nu-form}, we have
	$$|\del{a}(t/\nu)|= \Big|\frac{t}{\nu} \frac{\del{a} \nu}{\nu}\Big|
	= \Big|\Big(\frac{t}{\nu}\Big)^2 \frac{v_i \del{a}(v^i)}{\sqrt{1+|v|^2}}\Big|.
	$$
	Using the transformation $v^j = b(\psi) z^i$, we get
	$$ |\del{a}(t/\nu)\cdot  (A_0^0)^{-1} A^k(Z)|_{\op} \lesssim |z|^2 + |Dz|^2.
	$$
	The other part of $\Pbb \del{a}M^k(Z)\Pbb $ is simple to estimate and in total
	\begin{equation*}
		|\Pbb \del{a}M^k(Z)\Pbb |_{\op} \lesssim |z| + |Dz|.
	\end{equation*}
	 The remaining $ \Pbb $ and $ \Pbb^{\perp} $ conjugations of $ \partial_{a}M^{k} $ are easier to estimate. The estimates on $\partial_a M_0$ and \eqref{MLRW-estsFH}  use \eqref{MLRW:M0} and \eqref{MLRW-estsF} respectively. 
\end{proof}

\subsection{Extended system}
In order to apply the global existence result from \cite{FOW:2021} we wish to consider the Fuchsian system for $\bar{Z} \define (Z, DZ)^T$. Thus, we apply the differential operator $M^0 \del{a} (M^0)^{-1}$ to \eqref{conf-Eul-F} to get 
$$M_0(Z)\del{t} \del{a} Z + \frac1t (\mathcal C^k + M^k(Z)) \del{k}\del{a}Z =
\frac1t\Bc \Pi\del{a} Z
+ \frac1t G_a(\bar{Z}),
$$
where we introduce
\begin{align*}
	G_a(\bar{Z}) &:=  M_0(Z) \Big[
	-\del{a}\left[ M_0(Z)^{-1} \right](\mathcal C^k +M^k(Z))\del{k}Z 
	- M_0(Z)^{-1} \left( \del{a}M^k(Z)\right) \del{k} Z
	\\&\qquad\qquad
	+ \del{a}\left[ M_0(Z)^{-1} \right] \Bc \Pi Z
	+\del{a}\left[ M_0(Z)^{-1} F(Z)\right] \Big].
\end{align*}

\begin{lem}\label{conditions:relEuler}
	The extended system $\bar{Z} \define (Z, DZ)^T$ is governed by the PDEs:
	\begin{equation}\label{fixed-EoM}
		B^0(\bar{Z}) \del{t}\bar{Z}+\frac1t(C^k + B^k(\bar{Z}))\del{k}\bar{Z} = \frac1t \Bc(\bar{Z})\Pbb \bar{Z} + \frac1t H(\bar{Z}),
	\end{equation}
where
\begin{align*}
	B^0(\bar{Z}) \define \begin{pmatrix}
		M_0(Z) & 0 \\ 0 & M_0(Z)
	\end{pmatrix},
	\quad
	C^k \define  \begin{pmatrix}
		\mathcal  C^k & 0 \\ 0 & \mathcal C^k
	\end{pmatrix},
	\quad
	B^k(\bar{Z})\define \begin{pmatrix}
		M^k(Z) & 0 \\ 0 & M^k(Z)
	\end{pmatrix},
	\intertext{and}
	\Pbb \define \begin{pmatrix}
		\Pi & 0 \\ 0 & \Pi 
	\end{pmatrix},
	\quad 
	\mathcal B(\bar{Z}) \define \begin{pmatrix}
		\Bc & 0 \\ 0 & \Bc
	\end{pmatrix}, 
	\quad
	H(\bar{Z}) \define \begin{pmatrix}
		F(Z) \\ G(Z, DZ)
	\end{pmatrix}.
\end{align*}
Moreover, let $ \delta>0 $ and assume that $ |\bar{Z}|\leq \delta $. Then, the following estimates hold: 	\begin{align}
		\notag
		|\Pbb (B^0(\bar{Z})- B^0(0)) \Pbb|_{\op} + |\Pbb^\perp(B^0(\bar{Z})- B^0(0))\Pbb^\perp|_{\op}  &\label{est:B0}\\
		+ |\Pbb^\perp B^0(\bar{Z}) \Pbb|_{\op} +|\Pbb B^0(\bar{Z}) \Pbb^\perp|_{\op} &\lesssim |\Pbb \bar{Z}|^2,\\
		|\Pbb H(\bar{Z})| &\lesssim  |\Pbb \bar{Z}|,\label{est:PH}\\
		|\Pbb^\perp H(\bar{Z})| &\lesssim |\Pbb \bar{Z}|^2 ,\label{est:PperpH}\\
		|\Pbb^\perp B^k(\bar{Z})\Pbb|_{\op} + |\Pbb^\perp B^k(\bar{Z})\Pbb^\perp|_{\op}+ |\Pbb B^k(\bar{Z})\Pbb^\perp|_{\op}& \lesssim |\Pbb \bar{Z}|^2,\label{est:BK1}\\
		|\Pbb B^k(\bar{Z})\Pbb|_{\op}& \lesssim 
		|\Pbb \bar{Z}|,\label{est:BK2}\\
		|\Pbb \DIV B(t,v,w)\Pbb|_{\op}&\lesssim |t|^{-1} ,\label{est:div1}\\
		|\Pbb \DIV B(t,v,w)\Pbb^\perp |_{\op} + |\Pbb^\perp \DIV B(t,v,w)\Pbb|_{\op} &\lesssim |t|^{-1}  |\Pbb v|,\label{est:div2}
		\\
		|\Pbb^\perp \DIV B(t,v,w)\Pbb^\perp|_{\op}&\lesssim |t|^{-1} |\Pbb v|^2,\label{est:div3}
	\end{align}
 	where
 	\begin{equation*}
 			\DIV B(t,v,w) \define D_v B^0(v) \cdot (B^0(v))^{-1} \left( - \frac1t (C^k + B^k(v))w_k + \frac1t \Bc(v)\Pbb v + \frac1t H(v)
 		\right)
 		+ \frac1t D_v B^k(v) w_k.
 	\end{equation*}
    Following the conventions in \cite{beyeroliynyk}, we use $ v $ and $ w $ which denote $ \Bar{Z} $ and $ D \Bar{Z} $ respectively. 
 	\begin{comment}
	\begin{equation}\label{MLRW-ests1}
		\begin{split}
			|B^0(Z)|_{\op} + | C^k|_{\op} + |(B^{0})^{-1}(Z)|_{\op} 
			&\lesssim 1
			\\
			|B^k(Z)|_{\op} 
			&\lesssim |\Pi Z|
			\\
			|\del{a}B^k(Z)|_{\op}  + |\del{a}(B^0)^{-1}(Z)|_{\op} 
			&\lesssim |\Pbb \bar{Z}|
		\end{split}
	\end{equation}
	as well as
	\begin{equation}\label{MLRW-estsFH}
		|F(Z)| \lesssim  |\Pbb Z|^2 \AND 
		|G|\lesssim |\bar{Z}||\Pbb \bar{Z}|
	\end{equation}
	\end{comment}
\end{lem}

\begin{proof}
	\eqref{est:B0} immediately follows from the definition of $ B^{0} $ and $ M^{0} $ given in \eqref{MLRW:M0}. For \eqref{est:PH} we use \eqref{MLRW-estsFH} for the estimate on $ F $. To estimate $ G $, we define
	\begin{equation} \begin{split} \label{MLRW-comms1}
			C_1:= [\Pi, M_0]& = \begin{pmatrix} 0 & - A^0_j \\ A^0_i & 0 \end{pmatrix} , \quad
			C_2:= [\Pi, M_0^{-1} ] = - M_0^{-1}  ([\Pi, M_0])M_0^{-1}, \AND
			C_3^k := [\Pi^\perp, M^k]\,.
	\end{split}\end{equation}
	A computation shows that
	\begin{equation*}
		\begin{split}
			|\del{a}C_2|_{\op} 
			&\lesssim |z|^2 + |Dz|^2 \lesssim |\Pbb \bar{Z}|^2\,,
			\\
			|C_1|_{\op}  + |C_2|_{\op}  + |C_3|_{\op}  
			&\lesssim |z|^2\,.
		\end{split}
	\end{equation*}
	Using \eqref{MLRW-comms1} and \eqref{MLRW-ests2}, we find 
	\begin{align*}
		\Pi^\perp G_a&=  C_1 G_a 
		+ M_0 \left[ - \del{a}C_2 \cdot (\mathcal C^k+M^k)\del{k}Z - C_2 (\del{a}M^k)\del{k}Z
		+ \del{a}C_2 \cdot \Bc \Pi Z\right]
		\\&
		+ M_0 \Big[
		-\del{a}\left( M_0^{-1} \right)(\Pi^\perp \mathcal C^k +\Pi^\perp M^k)\del{k}Z 
		- M_0^{-1} \del{a}\left( \Pi^\perp M^k\Pi^\perp\right) \del{k} Z
		\\&\qquad\qquad
		- M_0^{-1} \del{a}\left( \Pi M^k\Pi^\perp\right) \del{k} \Pi Z
		+\Pi^\perp \del{a}\left( M_0^{-1} F\right) \Big].
	\end{align*} 
 	Note that it is crucial that the term $ M_{0}^{-1}\del{a}\left( \Pi M^k\Pi^\perp\right) \del{k} \Pi Z $ is of higher order in $ \Pbb Z $. The estimate \eqref{est:BK1} and \eqref{est:BK2} follow immediately from the estimates in \eqref{MLRW-ests1}. Using \eqref{MLRW-ests1}, \eqref{MLRW-estsFH} and \eqref{MLRW-ests2} we may conclude \eqref{est:PH} as well as \eqref{est:PperpH}. 
 	
 	Finally, we study the map
 	$$
 	\text{div}B(t,v,w) = D_v B^0(v) \cdot (B^0(v))^{-1} \left( - \frac1t (C^k + B^k(v))w_k + \frac1t \Bc(v)\Pbb v + \frac1t H(v)
 	\right)
 	+ \frac1t D_v B^k(v) w_k.
 	$$
 	We compute 
 	$$|D_v B^0(v)| \lesssim  |D_Z(f(\psi)\cdot |z|^2_{\gc})| \lesssim |\Pbb v|,$$
 	where $f$ is some smooth bounded function that comes from the choice of transformation. 
 	Also, $B^0 (v)^{-1} = \id + \Ord\bigl(|\Pbb v|^2\bigr)$. 
 	Together with \eqref{MLRW:Mk-ests}, \eqref{MLRW-ests1}, \eqref{est:PH} as well as Remark \ref{M0component}, this yields the desired results as given in \eqref{est:div1}--\eqref{est:div3}.
\end{proof}

\begin{prop}
	Suppose that $ k\geq 4 $ and $ \bar{Z}_{0}\in H^{k}(\mathbb{T}^{3}) $ is initial data to the initial value problem posed by \eqref{fixed-EoM}. Then, there exists an $ \epsilon>0 $ such that if $ \|\bar{Z}_{0}\|_{H_{k}}<\epsilon $, there exists a unique global solution $\bar{Z}$ to \eqref{fixed-EoM} with  
	\begin{equation*}
		\bar{Z}\in C^{0}((0,T_{0}],H^{k}(\mathbb{T}^{3}))\cap C^{1}((0,T_{0}],H^{k-1}(\mathbb{T}^{3})).
	\end{equation*}
\end{prop}
\begin{proof}
	After a transformation $ t\to -t $, our only task is to check the assumptions from  \cite[\textsection 4.1]{FOW:2021}. 
	It is straightforward to check that
	\begin{equation*}
		\Pbb^2 = \Pbb, \quad \Pbb^{\tr}=\Pbb, \quad \del{t} \Pbb = 0, \quad\del{j} \Pbb = 0\,,
	\end{equation*}
	and
	\begin{equation*}
		(C^k)^{\tr} = C^k, \quad\del{t}C^k = 0, \quad\del{i} C^k = 0\,.
	\end{equation*}
	Using that $0<K < \frac13$, there exist constants $\gamma_1, \gamma_2, \gamma_3>0$ such that 
	\begin{equation*}
		\frac{1}{\gamma_1} \id \leq M_0(Z) \leq \frac1\gamma_3\Bc(Z) \leq \gamma_2\id.
	\end{equation*}
	Since $A^0$ is symmetric, we have $(B^0(\bar{Z}))^{\tr} = B^0(\bar{Z})$. Also
	\begin{equation*}
		[\Pbb, \Bc(\bar{Z})] =0,\qquad 
		\del{k}(\Pbb \Bc(\bar{Z})) = 0,\AND
		|\Pbb \Bc(\bar{Z})- \Pbb \Bc(0)|_{\op} = 0 \,.
	\end{equation*}
	Since $A^k$ and $C^k$ are symmetric, $B^k$ is also. 	Together with Lemma \ref{conditions:relEuler} we conclude that all the conditions in \cite[\textsection 4.1]{FOW:2021} are met. 
\end{proof}

\section{The Einstein Euler equations near a Milne background}\label{EinEulBR}
We now consider solutions to the Einstein Euler equations \eqref{Ein-PF} allowing for the dynamical metric $\phFmet$ to be a perturbation away from the following linearly expanding Einstein spacetime: 

\begin{Def}[The Milne spacetime]
Let $(M, \gamma)$ negative closed Einstein space of dimension 3 and $\Ric[\gamma]=-\frac29 \gamma$. Then, the Lorentz cone $\mathcal{M}= \mathbb{R}\times M$ with metric
\begin{equation}
	\label{background}g_M = -d\bar{t}^2 + \frac{\bar{t}^2}{9}\gamma
\end{equation}
is a Lorentzian solution to the vacuum Einstein equations. We term $(\mathcal{M}, g_M)$  the (compactified) \textit{Milne} spacetime and refer to $\bar{t}$ as cosmological (or physical) time. 
\end{Def}

In this section we prove the following theorem

\begin{thm}\label{thm:Milne_stability}
	Let $ (\mathcal{M}, g_M) $ be the Milne spacetime and consider the Einstein-Euler equations \eqref{Ein-PF} with linear equation of state \eqref{eos-lin} where $K \in (0, 1/3)$. 
	Let $ (g_{0},k_{0},\rho_{0},u_{0}) $ be initial data satisfying the constraint equations \eqref{EoM-constraints} for at physical time $ t_{0} $ such that $\rho_0>0$ and
	\begin{equation}
		(g_{0},k_{0},\rho_{0},u_{0})\in \mathscr{B}^{6,5,5,5}_{\epsilon}\big(\tfrac{t_{0}^{2}}{9}\gamma,-\tfrac{t_{0}}{9}\gamma,0,0\big).\label{initialsmallness}
	\end{equation}
	Then, there exists an $ \epsilon>0 $  sufficiently small, such that the future development of $ (g_{0},k_{0},\rho_{0},u_{0}) $ under the Einstein-Euler equations  is complete and admits a constant mean curvature foliation labelled by $ \tau\in[\tau_{0},0) $, such that the induced metric and second fundamental form on constant mean curvature slices converge as
	\begin{equation*}
		(\tau^{2}g,\tau k)\overset{\tau \to 0}{\longrightarrow} \big(\gamma,\frac{1}{3}\gamma\big).
	\end{equation*}
\end{thm}

\begin{rem}
	In theorem \ref{thm:Milne_stability} above  $ \mathscr{B}_{\epsilon}^{6, 5, 5, 5}(\cdot, \cdot, \cdot, \cdot) $ denotes the ball of radius $\epsilon$ centered at the argument in the set $ H^{6}(M)\times H^{5}(M)\times H^{5}(M)\times H^{5}(M) $, with the canonical Sobolev norms given in Definition \ref{Sobolevspaces} below. 
\end{rem}

\begin{Def}[ADM decomposition and CMCSH gauge]
We reparametrise the physical dynamical metric $\phFmet$ in terms of the ADM variables via 
\begin{equation}\label{399}
\phFmet = -\tlapse^2 d t^2 +\phTmet_{ab}(d x^a+\tshift^ad t)(d x^b + \tshift^b d t).
\end{equation}
That is, $\tlapse$ is the lapse, $\tshift$ is the shift and $\phTmet$ is the induced Riemannian metric on $M$. 
Furthermore, we denote the mean curvature and trace-free part of the second fundamental form\footnote{We assume that all spatial indices are raised and lowered using the metric $ g $. } by
$$
\tau\define \text{tr}_{\tilde g} \tilde k = \phTmet^{ab} \tilde{k}_{ab},\qquad
\tilde k_{ab}\define \tilde\Sigma_{ab}+\tfrac13 \tau \tilde g_{ab}.
$$
	We follow the standard approach for the Milne spacetime \cite{AM03, AnderssonMoncrief:2011} and impose the \emph{constant mean curvature}
and \textit{spatial harmonic gauge} conditions
\begin{equation}\label{eq:spatial_harmonic}
t=\tau, \qquad H^a:= \tilde{g}^{cb}(\Gamma[\tilde{g}]^{a}_{cb}-\Gamma[\gamma]^a_{cb})=0.
\end{equation}
\end{Def}

\begin{Def}[Rescaled geometric variables]
We define
\begin{equation*}
g_{ab}\define\tau^2\tilde{g}_{ab}, \quad
g^{ab}\define(\tau^2)^{-1}\tilde{g}^{ab},
\quad
N\define\tau^2\tlapse, \quad 
X^a\define\tau \tshift^a,\quad
\Sigma_{ab}\define\tau\tilde \Sigma_{ab},
\end{equation*}
as well as $\hat{N}\define N-3$. 
We also define the \emph{logarithmic time} $ T $ as $T\coloneqq -\ln(\frac{\tau}{e\tau_{0}})$ so that $T$ satisfies the relation $\partial_T = - \tau \partial_\tau$.
\end{Def}

\begin{rem}
The expression for the Milne geometry given in \eqref{background} is with respect to cosmological time $ \bar{t} $. Moving to CMC time, given by $ \tau =-3/\bar{t}$, we see that the metric takes the form 
$$
g_M = 
\frac{1}{\tau^2}\left( - \frac{9}{\tau^2}d\tau^2 + \gamma\right).$$
Thus, the Milne spacetime, when written in CMCSH-gauge and rescaled variables, is given by
	\begin{equation*}
		(g_{ij},\Sigma_{ij},N,X^{i})=(\gamma_{ij},0,3,0).
	\end{equation*}
\end{rem}

We also define various components of the energy momentum tensor that contribute to the dynamics under the Einstein flow, see \cite{AF20}. 

\begin{Def}[Matter variables $E, \jmath, \eta, S$]\label{matterquant}
	We define the following physical matter quantities 
	\begin{align*}
		\tilde{E}\define \Tb^{\mu\nu}n_{\mu}n_{\nu}, \quad 
		\tilde{j}^{a}\define\tilde{g}^{ab}\tilde{N}\Tb^{0\mu}\bar{g}_{b\mu}, 
		\quad
		\tilde{\eta}\define\tilde{E}+\tilde{g}^{ab}\Tb_{ab}, 
		\quad \tilde{S}_{ab}\define\Tb_{ab}-\frac{1}{2}\tr_{\bar{g}}\Tb\cdot \tilde{g}_{ab},
	\end{align*}
and the rescaled matter quantities by
	\begin{align*}
		E\define(-\tau)^{-3}\tilde{E},
		\quad
		j^{a}\define(-\tau)^{-5}\tilde{j}^{a},
		\quad
		\eta\define (-\tau)^{-3} \tilde{\eta},
		\quad
		S_{ab}\define (-\tau)^{-1}\tilde{S}_{ab}.
	\end{align*}
\end{Def}

\begin{Def}[Elliptic operators $\Delta_{g,\gamma}, \mathscr{L}_{g, \gamma}$]
	Let $ V $ be a symmetric $ (0,2) $-tensor on $ M $. We define the following self-adjoint, elliptic differential operators
	\begin{equation*}
	\begin{aligned}
	\Delta_{g,\gamma} V_{ij}&\coloneqq \big(\sqrt{g}\big)^{-1}\nabla[\gamma]_{a}\big(\sqrt{g}g^{ab}\nabla[\gamma]_{b}V_{ij}\big),
	\\
	\mathscr{L}_{g,\gamma}V_{ij}&\coloneqq -\Delta_{g,\gamma}V_{ij}-2\Riem[\gamma]_{iajb}\gamma^{ac}\gamma^{bd}V_{cd}.
	\end{aligned}
	\end{equation*}
%	Moreover the operator $ \mathscr{L}_{g,\gamma} $ is self-adjoint with respect to the \emph{mixed $ L^{2} $-product}, i.e.~ 
%	\begin{equation*}
%	\begin{aligned}
%	(\mathscr{L}_{g,\gamma}V,W)_{L^{2}(g,\gamma)}= \int_{M}\langle \mathscr{L}_{g,\gamma}V,W\rangle_{\gamma} =(V,\mathscr{L}_{g,\gamma}W)_{L^{2}(g,\gamma)}. 
%	\end{aligned}
%	\end{equation*}
\end{Def}

The equations of motion for an Einstein-matter system in CMCSH gauge are derived in \cite{AF20}. Denoting the  Levi-Civita connection of $g$ by $ \nabla$, the PDEs are as follows. 

\begin{lem}[Equations of motion for the geometric variables]
The Einstein equations \eqref{Ein} reduce to two constraint equations
\begin{subequations}\label{EoM}
	\begin{equation}\begin{aligned}\label{EoM-constraints}
			\text{R}(g)-|\Sigma|_g^2+\tfrac{2}{3}&= 
			-4 \tau E ,
			\\
			\nabla^a \Sigma_{ab} &=
			2 \tau^2 \jmath_b ,
	\end{aligned}\end{equation}
two elliptic equations for the lapse and shift variables
	\begin{equation*}\begin{aligned}
			(\Delta - \tfrac{1}{3})N &= 
			N \left( |\Sigma|_g^2 - \tau \eta \right)-1, 
			\\
			\Delta X^a + \Ric[g]^a{}_b X^b &=
			2 \nabla_b N \Sigma^{ba} - \nabla^a \hat{N}+ 2 N \tau^2 \jmath^a 
			- (2N \Sigma^{bc} - \nabla^b X^c)(\Gamma[g]^a_{bc} -  \Gamma[\gamma]^a_{bc}),
	\end{aligned}\end{equation*}
and two evolution equations for the induced metric and trace-free part of the second fundamental form 
	\begin{equation}\begin{aligned}\label{eq:EoM-pT-g-Sigma}
			\partial_T g_{ab} &=
			2N \Sigma_{ab} + 2\hat{N} g_{ab} - \mathcal{L}_X g_{ab}, 
			\\
			\partial_T \Sigma_{ab} &=
			-2\Sigma_{ab} - N(\Ric[g]_{ab} +\tfrac{2}{9}g_{ab} ) + \nabla_a \nabla_b N + 2N \Sigma_{ac} \Sigma^c_b 
			\\
			& \quad 
			-\tfrac{1}{3} \hat{N}  g_{ab} - \hat{N} \Sigma_{ab} - \mathcal{L}_X \Sigma_{ab} + N \tau S_{ab}.
	\end{aligned}\end{equation}
\end{subequations}
\end{lem}

\begin{rem}
In the above equations \eqref{EoM}, $\mathcal{L}_X$ denotes the Lie derivative with respect to $X$. We also recall  from \cite{AM03} the following decomposition of the curvature term in the spatially harmonic gauge:
\begin{equation}\label{eq:Ricci-decomp}\Ric[g]_{ab} +\frac29 g_{ab} = \frac12 \mathscr{L}_{g,\gamma}g_{ab} +J_{ab},
\end{equation}
where there is a constant $C>0$ such that $ \| J\|_{H^{s-1}} \leq C \| g-\gamma\|_{H^s}$. 
\end{rem}

\subsection{Equations of motion for the fluid variables}
In CMCSH-coordinates and rescaled variables the dynamical metric \eqref{399} is given by
\begin{equation*}\phFmet = \frac{1}{\tau^2} \left[ - \tfrac{N^2}{\tau^2} d\tau^2 + g_{ab} (dx^a+ \tfrac{X^a}{\tau}d\tau)(dx^b+\tfrac{X^b}{\tau} d\tau)
\right].
%=: \frac{1}{\tau^2}g
\end{equation*}
This is in the form considered in \eqref{3+1-g}, provided we make the identifications
$$ x^0 \equiv \tau, \quad \alpha \equiv \frac{N}{\tau}, \quad \beta^a \equiv \frac{X^a}{\tau}, \quad g_{ab} \equiv \gc_{ab}, \quad \Psi(\tau) = \ln(-\tau).
$$

\begin{lem}
The Euler equations \eqref{rel-Eul} with respect to the dynamical metric $\phFmet$ can be written as
\begin{equation} \label{Milne_rel_Eul}
 B^0 \del{\tau} U + B^k\nabla_k  U =  H,
\end{equation}
where $ U = (\zeta, u^j)^{\tr}$ and 
\begin{equation*}
 B^0 = \begin{pmatrix}K & -\frac{K}{ \nu \mu} w_j \\
-\frac{K}{ \nu \mu} w_i &  M_{ij} \end{pmatrix}, 
\quad
 B^k = \begin{pmatrix}  \frac{K}{ \nu} u^k & \frac{s^2}{ \nu}\delta^k_j \\
\frac{K}{ \nu} \delta^k_i & \frac{1}{ \nu}  M_{ij}  u^k \end{pmatrix},
\quad
 H= \begin{pmatrix} K  \ell \\
\Bigl(\frac{1-3 K }{ \nu \mu} w_i\Bigr)+ m_i\end{pmatrix}.
\end{equation*}
\end{lem}
\begin{proof}
Recall from \eqref{conform-decomp} that $
\nu = v^\tau, \; w_j =v_j, \; \mu = v_\tau$ and $u^j = v^j$. 
Evaluating the expressions given in \eqref{3+1-Christ-ij0}, \eqref{Mij-3+1}, \eqref{ell-def} and \eqref{mi-def}  in CMCSH-coordinates, we compute	the geometric quantities to be
\begin{equation}\label{345a}\begin{split}
\Kc_{ij} &= \frac{1}{2N}(\del{T}g_{ij} + \nabla_i X_j+\nabla_j X_i),
\\
\Xi^i_j&=\frac{1}{\tau}\Big[-\frac{1}{N}X^i(\nabla_j N -X^k\Kc_{kj})-N \Kc_j^i +\nabla_j X^i\Big],
\\
\Upsilon^i&= \frac{1}{\tau^2}\Big[N \nabla^i N - 2NX^j \Kc_j^i-\frac{X^i }{N}(-\del{T}N-N+X^j\nabla_j N-X^jX^k\Kc_{jk})
	\\&\qquad\qquad
	-\del{T}X^i - X^i+X^j\nabla_jX^i \Big],
\end{split}\end{equation}
while the matter quantities are
\begin{equation}\label{345b}\begin{split}
M_{ij} &=g_{ij}-\frac{1}{\tau \mu}(w_iX_j +w_j X_i) + \frac{-N^2+|X|_{g}^2}{\tau^2 \mu^2}w_i w_j,
\\
\ell &= -\frac{1}{\nu N}\Kc_{jk}X^j u^k-\Xi^j_j +\frac{1}{\mu}\Upsilon^j w_j + \frac{1}{\nu\mu}\Xi^j_k w_j u^k ,
 \\
m_i &= - M_{ij}\Bigl( \frac{1}{\nu N}\Kc_{lk}X^j u^l u^k + \nu \Upsilon^j +2 \Xi^j_k u^k\Bigr).
\end{split}\end{equation}
The result then follows from Lemma \ref{lem:conf-Eul-U}. 
\end{proof}

\begin{rem}
On the background Milne geometry the tensors $\Kc_{ij}, \tau \Xi^i_j$ and $\tau^2 \Upsilon^i$ identically vanish. 
\end{rem}

\begin{lem}
The homogeneous fluid solutions to \eqref{rel-Eul} are given by $U=(\zeta, v^a)=0$. 
\end{lem}
\begin{proof}
The modified density variable as defined in Section \ref{sec:transf} takes the form 
\begin{equation}\label{zeta:dynamic}
\zeta = \frac{1}{1+K} \ln(\trho/\trho_0) - 3 \ln(-\tau).
\end{equation}
 Considering the Euler equations  \eqref{rel-Eul} with $\phFmet=g_M$ and  $\vb^i = 0$ (i.e.~a homogeneous regime) we arrive at the ODEs
 \begin{comment}
\begin{align*}
\vb^{\bar{t}}\partial_{\bar{t}}\trho + (1+K) \trho\partial_{\bar{t}}\vb^{\bar{t}}+ \frac{3(1+K)}{\bar{t}}\trho \vb^{\bar{t}}=0, \quad
\trho\vb^{\bar{t}}\partial_{\bar{t}}\vb^{\bar{t}} + \frac{K}{1+K} \left( (\vb^{\bar{t}})^2 - 1\right) \partial_{\bar{t}}  \trho = 0
\end{align*}
\end{comment}
\begin{equation*}
	\big((\bar{v}^{\bar{t}})^{2}-K/(1+K)\big)\partial_{t}\bar{\rho}+2\bar{\rho} \bar{v}^{\bar{t}}\partial_{t}\bar{v}^{\bar{t}}+\frac{3}{t}(\bar{v}^{\bar{t}})^{2}\bar{\rho}=0,\qquad \partial_{i}\bar{\rho}=0.
\end{equation*}

These are satisfied by 
$$\vb^{\bar{t}}=1 \AND \trho = \trho_0 \bar{t}^{-3(1+K)}$$ 
for some constant $\trho_0>0$. This implies $\trho = c_h (-\tau)^{3(1+K)}$ for some constant $c_h>0$ and thus on the background
$$ \zeta = \frac{1}{1+K} \ln (\trho/\trho_0)- 3\Psi=3 \ln (-\tau) - 3 \Psi= 0,$$
provided we pick $\trho_0 = c_h$. Hence, our homogeneous solution is $U=(\zeta, v^a)=0$. 
\end{proof}

\begin{rem}
On the above homogeneous fluid solutions, we have the simplifications
\begin{align}\label{nu-background}
\C\nu = v^\tau= -\frac{\tau}{3},
\quad \C\mu=v_\tau =\frac{3}{\tau}, 
\quad
\quad  \C\nu\C\mu = -1, \quad M_{ij}  = \gamma_{ij}.
\end{align}
Additionally, $ w_j, u^j, w^j, \tau \ell$ and $\tau m_i$ identically vanish. 
\end{rem}

Due to the factor of $\tau$ appearing in the expression for $\nu$ in \eqref{nu-background}, we introduce an additional rescaling.
\begin{Def}[$\hat{v}^\tau$] We rescale the time-component of the conformal fluid four-velocity by
$$ \hat{v}^{\tau}\define (-\tau)^{-1}v^{\tau}.$$ 
\end{Def}

\subsection{Preliminary notation and computations}
In this subsection we assume that $ (M,g) $ is a closed $ 3 $-dimensional Riemannian manifold with Levi-Civita connection $\nabla$.

\begin{Def}[Inner products] 
	Let $ v $ and $ w $ be vectors fields on $M$.  We define $ \langle \cdot,\cdot \rangle_g$ as
	\begin{equation*}
		\langle v,w \rangle_{g} \coloneqq g_{ij}v^{i}w^{j}.
	\end{equation*}
	If $ \ell\geq 1 $, then we define
	\begin{align*}
		\langle \nabla^{\ell}v, \nabla^{\ell}w\rangle_{g}&\coloneqq g^{a_{1}b_{1}}\cdots g^{a_{\ell}b_{\ell}}\langle\nabla_{a_{1}}\cdots\nabla_{a_{l}}v,\nabla_{b_{1}}\cdots\nabla_{b_{\ell}}w\rangle_{g}.
	\end{align*}
	This definition is extended in the usual way to arbitrary tensor fields. 
	We also define the modulus as $ |v|_{g}^{2}\coloneqq \langle v,v \rangle_{g} $.
	If the subsript is omitted, we assume the bracket to be the Euclidean inner product, i.e.~$\langle v,w \rangle \coloneqq v^{T}(w) $.
\end{Def}

\begin{Def}[Sobolev norms and measures]\label{Sobolevspaces}
We denote  by $ \mu_{g} $ the Riemannian measure associated to $ g $, which is given locally by $ \sqrt{\det g}dx^{1}\wedge dx^{2} \wedge dx^{3} $. When the context is unambiguous, we suppress the measure and write
\begin{equation*}
	 \int_{M} f = \int_{M} f \mu_{g}
\end{equation*}
for $f$ some function. 
	For a tensor field $ V $ and $s\in\mathbb{N}$, we define the \emph{Sobolev norm of order $s$} as
	\begin{align*}
		\|V\|_{H^{s}}^{2}&\coloneqq \sum_{0 \leq \ell\leq s}\int_{M}|\nabla^{\ell}V|_{g}.
	\end{align*}
	We frequently consider the norms of an abstract vector quantity $ \mathcal{V}=(f,v)^{T} $ which consists of a function $ f $ and a spatial vector field $ v $. We simply write 
	\begin{equation*}
		\|\mathcal{V}\|_{H^{k}}\coloneqq \|f\|_{H^{k}}+\|v\|_{H^{k}}.
	\end{equation*}
\end{Def}

We also require the following result that allows us to commute the time-derivative with integration. 

\begin{lem}\label{timederivative}
	For an arbitrary scalar function $ f $ the following identity holds: 
	\begin{equation*}
		\frac{d}{d T} \int_{M}f\; \mu_g \lesssim \|N-3\|_{H^{2}}\|f\|_{L^{1}}+\|X\|_{H^{3}}\|f\|_{L^{1}}+\int_{M}\partial_{T}f\; \mu_g.
	\end{equation*}
\end{lem}

\begin{proof}
	The following identity from \cite{ChoquetBruhatMoncrief:2001}, together with integration by parts on the shift term, yields
	\begin{align*}
		\frac{d}{dt}(\int_{M}f\; \mu_g)&=\int_{M}\big(3(N-3)f+\partial_{T}f-f\nabla_{i}X^{i})\big)\mu_{g}.
	\end{align*}
The result then follows by the Hölder inequality and the  Sobolev embedding $L^\infty(M) \hookrightarrow H^2(M)$. 
\end{proof}

We also state an additional lemma which will be needed later on, when discussing the behavior of the energy functionals. 

\begin{lem}\label{Riemannestimate}
Suppose that $\gamma$ is another Riemannian metric on $M$, such that $\Ric[\gamma]=-\frac29 \gamma$, the harmonic condition \eqref{eq:spatial_harmonic} holds and also the bounds
\begin{equation*}
\| g-\gamma\|_{H^s} \leq C \epsilon
\end{equation*}
hold for $s\in\mathbb{N}$.
Then 
	\begin{equation*}
		\| {\rm{Riem}}[g]-{\rm{Riem}}[\gamma]\|_{H^{s}}\lesssim \|g-\gamma\|_{H^{s+2}}.
	\end{equation*}
\end{lem}

\begin{proof}
	By the Ricci decomposition of the Riemann tensor in 3 dimensions, we have that
	\begin{align}
		\Riem[\gamma]_{ijkl}&=-\frac{\R[\gamma]}{6}(\gamma_{il}\gamma_{jk}-\gamma_{ik}\gamma_{jl})=\frac{1}{9}(\gamma_{il}\gamma_{jk}-\gamma_{ik}\gamma_{jl}),\notag\\
		\Riem[g]_{ijkl}&=-\frac{\R[g]}{6}(g_{il}g_{jk}-g_{ik}g_{jl})+(V_{il}g_{jk}-V_{jl}g_{ik}-V_{ik}g_{jl}+V_{jk}g_{il}),\label{Riemann}
	\end{align}
	where $ V_{jk}\define-\Ric[g]_{jk}+\frac{1}{3}\R[g]g_{jk} $ and we used that $\Ric[\gamma]=-\frac{2}{9}\gamma $. 
From \eqref{eq:Ricci-decomp}, we have that
	\begin{equation*}
	 \Ric[g]_{ab} -\Ric[\gamma]_{ab}= -\frac29 g_{ab} + \frac12 \mathscr{L}_{g,\gamma}(g_{ab}-\gamma_{ab}) +J_{ab} + \frac29 \gamma_{ab}.
	\end{equation*}
Thus, by elliptic regularity (see e.g. \cite{Besse})
	\begin{align*}
		\| \Ric[g]_{ab} -\Ric[\gamma]_{ab}\|_{H^{s}}&\lesssim \|g-\gamma\|_{H^s} + \| \mathscr{L}_{g,\gamma}(g-\gamma)\|_{H^s} + \| J\|_{H^s}  \lesssim \| g-\gamma\|_{H^{s+2}}.
	\end{align*}
	To compute the difference of the Riemann tensors, one then expand all of the terms in \eqref{Riemann} including $ g $ by $ g-\gamma +\gamma $ and $ \Ric[g] $ by $ \Ric[g]-\Ric[\gamma]+\Ric[\gamma] $. Applying the triangle inequality then shows that $ \|V\|_{H^{s}}\lesssim \|g-\gamma\|_{H^{s+2}} $ and furthermore the desired estimate. 
\end{proof}

\subsection{Local existence and bootstrap assumptions}
\begin{thm} 
	Let $ k\geq 6 $ be a fixed integer. At $ T=T_{0} $, suppose that we have CMC initial data satisfying the constraints \eqref{EoM-constraints} with regularity
	\begin{equation*}
		(g_{0},k_{0},N_{0},X_{0},\rho_{0},u_{0})\in H^{k}\times H^{k-1}\times H^{k}\times H^{k}\times H^{k-1}\times H^{k-1}.
	\end{equation*}
	Then, there exists a unique classical solution $ (g,k,N,X,\rho,u) $ to \eqref{EoM} on $ [T_{0},T_{+}) $ with $ T_{+}>T_{0} $. This local solution satisfies
	\begin{align*}
		g,N,X\in C^{0}([T_{0},T_{+}),H^{k})\cap C^{1}([T_{0},T_{+}),H^{k-1}),\\
		k\in C^{0}([T_{0},T_{+}),H^{k-1})\cap C^{1}([T_{0},T_{+}),H^{k-2}),\\
		u,\rho\in C^{0}([T_{0},T_{+}),H^{k-1}),\\
		\partial_{T}\rho,\partial_{T}u\in C^{0}([T_{0},T_{+}),H^{k-2}).
	\end{align*}
	In addition, the norms as well as the time of existence $ T_{+} $ depend continuously on the initial data. By the continuation principle the maximal time of existence $ T_{\text{max}} $ is either $ T_{\text{max}}=\infty $, i.e. global existence, or 
	\begin{align*}
		\lim_{T\to T_{\text{max}}}\sup_{[T_{0},T]}\|g-\gamma\|_{H^{k}}+\|\Sigma\|_{H^{k-1}}+\|N-3\|_{H^{k}}+\|X\|_{H^{k}}\\+\|\partial_{T}N\|_{H^{k-1}}+\|\partial_{T}X\|_{H^{k-1}}+\|\rho\|_{H^{k-1}}+\|u\|_{H^{k-1}}>\delta,
	\end{align*}
	where $ \delta>0 $ is a fixed constant.
\end{thm}

In light of the above local existence theorem, we can now make certain bootstrap assumptions that hold in a non-empty time interval. 
Let $\lambda<1$ be a fixed positive constant and, again, $k \geq 6$ a fixed integer. We assume that, for all $T_0 \leq T \leq T'$, where $T'>T_0$, there is a constant $C>0$ such that the following bootstrap assumptions hold 
	\begin{equation}\label{bootstrap}
		\begin{aligned}
			\|U\|_{H^{k-1}}&\leq C \epsilon, \\
			 \| g-\gamma\|_{H^{k}} + \| \Sigma\|_{H^{k-1}} &\leq C \epsilon e^{-\lambda T},
			\\
		  \|N-3\|_{H^{k}}+ \|X\|_{H^{k}} + \|\partial_T N\|_{H^{k-1}} + \|\partial_T X\|_{H^{k-1}} &\leq C \epsilon e^{-T}.\\
		\end{aligned}
	\end{equation}

The results in the rest of this paper will be derived under these bootstrap assumptions.

\begin{lem}
Under \eqref{bootstrap} we have that
	$$\| \hat{v}^{\tau} -\tfrac13\|_{H^k} \lesssim \epsilon.$$
\end{lem}

\begin{proof}
	The normalization condition (\ref{vb-norm}) implies 
	\begin{align*}
		\bar{v}^{\tau}&=\frac{1}{2(-\tilde{N}^{2}+\tilde{X}_{a}\tilde{X}^{a})}(-2\tilde{X}_{a}\tilde{v}^{a}+\big[4(\tilde{X}_{a}\tilde{X}^{a})^{2}-4(-\tilde{N}^{2}+\tilde{X}_{a}+\tilde{x}^{a})(\tilde{g}_{ab}\tilde{v}^{a}\tilde{v}^{b}+1)\big]^{\frac{1}{2}})\\
		&=\frac{\tau^{2}}{X_{a}X^{a}-N^{2}}(-X_{a}v^{a}+\big[(X_{a}v^{a})^{2}+(N^{2}-X_{a}X^{a})(g_{ab}v^{a}v^{b}+1)\big]^{\frac{1}{2}}),
	\end{align*}
and so the final estimate is obtained by applying \eqref{bootstrap}. 
\end{proof}

\begin{rem}
	From the geometric equations of motion \eqref{eq:EoM-pT-g-Sigma} we may immediately conclude that
\begin{equation}\label{estgtime}
	\|\partial_{T}g\|_{H^{k-1}}\lesssim \|N\Sigma\|_{H^{k-1}}+\|\hat{N}g\|_{H^{k-1}}+\|\mathscr{L}_{X}g\|_{H^{k-1}}\lesssim \epsilon  e^{- T}.
\end{equation}
\end{rem}
 
We now explain why the bootstrap estimate on $ U $ yields a similar estimate on $ Z $.

\begin{lem}
	Under the bootstrap assumptions \eqref{bootstrap} we have that
	\begin{equation*}
		\|Z\|_{H^{k-1}}\lesssim \epsilon. 
	\end{equation*}
\end{lem}
\begin{proof}
	Consider the transformation $ Z=(\psi, z^i) \mapsto U =(\xi, u^i)$ as given in \eqref{cov} which we will denote by $ \varphi:\mathbb{R}^{4}\to \mathbb{R}^{4} $. By the bootstrap \eqref{bootstrap} as well as Sobolev embedding we have that
	\begin{equation*}
		\|bz\|_{L^{\infty}}+\|a\|_{L^{\infty}}\lesssim \|U\|_{H^{k-1}}\lesssim \epsilon. 
	\end{equation*}
	Using the explicit form of $ |b| $, which is bounded from below, we can conclude that $ \|z\|_{L^{\infty}}\lesssim \epsilon $. Choosing $ c_{1,2} $ the same as in the proof of Lemma \ref{homsolutions} we see that smallness of $ a $ must imply that $ \|\psi\|_{L^{\infty}}\lesssim \epsilon $. Simply calculating the Jacobian of $ \varphi $, we see that due to smallness of $ Z $, $ \varphi $ is a diffeomorphism. Also note that $ z^{i}=f(\psi)u^{i} $ is a product of a scalar and a vector and is hence a vector and furthermore that due to the structure of $ a $, $ \psi $ then has to be a scalar as well. Let us now consider the $ \dot{H}^{1} $-norm of $ \psi $:
	\begin{align*}
		\|\psi\|_{\dot{H}^{1}}^{2}&=\int_{M}\big(g^{mn}\nabla_{m}(\psi(\zeta,|u|_{g}^{2}))\nabla_{n}(\psi(\zeta,|u|_{g}^{2}))\\
		&=\int_{M}\big(g^{mn}\big[(\partial_{1}\psi)(\zeta,|u|_{g}^{2})\partial _{m}\zeta+(\partial_{2}\psi)(\zeta,|u|_{g}^{2})\partial _{m}|u|_{g}^{2}\big]\big[(\partial_{1}\psi)(\zeta,|u|_{g}^{2})\partial _{n}\zeta+(\partial_{2}\psi)(\zeta,|u|_{g}^{2})\partial _{n}|u|_{g}^{2}\big]\big)\\
		&\lesssim \int_{M} \big(g^{mn}\nabla_{m}\zeta\nabla_{n}\zeta +g^{mn}u_{i}\nabla_{m}u^{i}u_{k}\nabla_{n}u^{k}+g^{mn}\nabla_{m}\zeta u_{i}\nabla_{n}u^{i}\big)\lesssim \|U\|_{H^{1}},
	\end{align*}
	where we used Cauchy-Schwarz and the fact that $ \partial_1\psi, \partial_2 \psi $ are scalar functions on a compact manifold and hence bounded. Since $ z^{i}=f(\psi)u^{i} $ we have, by the Leibniz rule of the covariant derivative,
	\begin{equation*}
		\|z\|_{\dot{H}^{1}}\lesssim \|U\|_{H^{1}}.
	\end{equation*}
	For higher derivatives the same calculation follows with more applications of the Leibniz rule.
\end{proof}

\subsection{Transformed Fuchsian system for the Euler equations}

In this section we derive the expression for the equations of motion of the fluid. 

\begin{lem}\label{lem:rough-sources}
	Under \eqref{bootstrap} the following estimates hold:
	\begin{equation*}\begin{split}
			|\tau \ell|  + |\tau m^i|_g+|\Kc_{ij}|_g+ |\tau \Xi^i_j|_g+|\tau^2 \C\Upsilon^i|_g &\lesssim |X|_g+|\nabla X|_g+|\hat{N}|+|\nabla N|_g+|\Sigma|_g+|\partial_{T}X|_g ,
			\\
			|M_{ij}-\gamma_{ij}|_g &
			\lesssim |g-\gamma|_g+|X|_g+ |z|_g^2.
		\end{split}
	\end{equation*}
\end{lem}
\begin{proof}
Using the expressions given in \eqref{345a} as well as the equation of motion of the geometry, \eqref{eq:EoM-pT-g-Sigma}, we obtain
	\begin{align*}
			|\Kc_{ij}|_g &\lesssim |\partial_T g |_g + |\nabla X|_g
			\lesssim |\nabla X|_g+|X|_g+|\hat{N}|+|\Sigma|_g ,
			\\
			|\tau \Xi^i_j|_g&\lesssim |\partial_T g |_g + |\nabla X|_g + |X \nabla N|_g
			\lesssim |\nabla X|_g + |\nabla N|_g+|X|+|\hat{N}|+|\Sigma|_g ,
			\\
			|\tau^2 \C\Upsilon^i|_g&\lesssim |\nabla N|_g + |\partial_T X|_g + |X|_g.
\end{align*}
Similarly, using the expressions given in \eqref{345b} we find
\begin{align*}
			|M_{ij}-\gamma_{ij}|_g &\lesssim |g-\gamma|_g + |X|_g^2|\nu| + |z|_g|X| + |z|_g^2
			\lesssim |g-\gamma|_g+|X|_g+ |z|_g^2,
			\\
			|\tau \ell|&\lesssim \left( |\tau/\nu| |\Kc|_g|X|_g  + |\tau^2 \Upsilon|_g|\tau\mu|^{-1}+|\tau\Xi|_g\right) |z^j|_g + \left(|\tau \Xi|_g+ |\tau^2 \Upsilon|_g \right)| X|_g|\nu|+ |\tau \Xi|_g,
			\\
			|\tau m^i|_g&\lesssim |M||z^j|_g\left( |\tau/\nu||\Kc|_g|X|_g|z^j|_g + |\tau \Xi|_g\right)+ |M| |\nu/\tau||\tau^2 \Upsilon|_g.
	\end{align*}
\end{proof}

\begin{lem}[Fluid equations of motion]
	The Euler equations \eqref{Milne_rel_Eul} can be rewritten as
	\begin{equation}\label{Milne:fluid-PDE}
		M^0 \partial_T Z - (C^k + M^k) \nabla_k Z = -\Bc \Pbb Z - F,
	\end{equation}
	where 
	\begin{align*}
M^0(Z) &= \begin{pmatrix}
1 & 0 \\ 0 & K^{-1} g_{ij}
\end{pmatrix} +  e^{-2T} \begin{pmatrix} 0 & 1 \\ 1 & 0 \end{pmatrix} + \Ord\bigl(|X|^2_g\bigr)+\begin{pmatrix}
0 & | z|_{g}^2 \\ | z|_{g}^2 & | z|_{g}^2
\end{pmatrix} ,
\\
M^k(Z) &= \begin{pmatrix} 0 & 0 \\
0 & \Ord(|z|_g)
    \end{pmatrix} +  \Ord\bigl(|X|^2_g + | z|_{g}^2\bigr), \AND
\end{align*}
\begin{align}\label{Milne-Ckeq}
\Bc &:= (K^{-1} -3)\id, \quad 
\Pbb := \begin{pmatrix} 0 & 0  \\0 & \delta^i_j\end{pmatrix}, 
\quad C^k 
= \begin{pmatrix} 0 & \delta^k_j \\
\delta^k_i & 0
    \end{pmatrix},
\end{align}
and
\begin{equation}\begin{split}\label{estF}
	|\Pbb F| +|\Pbb^\perp F| &\lesssim  \Ord( |z|_g^2+|X|_g+|\nabla X|_g+|\hat{N}|+|\nabla N|_g+|\Sigma|_g+|\partial_{T}X|_g).
\end{split}
\end{equation}
\end{lem}

\begin{proof}
We transform to new variables $Z=( \psi,  z^a)$ and define $z_i = g_{ij} z^j$.  
We first compute 
$$
 w_a =  v_a =  g_{a\mu}  v^\mu = X_a  v^\tau + g_{ac}   z^c b( \psi) = X_a \nu + b z_a \,.
$$
Using \eqref{conf-Eul-F} and \eqref{transf-derivs1}, we compute (using things like $\mu^{-1} \sim \tau \lesssim 1$)
\begin{align*}
    A^0_0 &= 
    b^2 K \left[ 1 + \frac{2}{\mu (\psi+c_2)} X_k z^k\right] +  \Ord\bigl(| z|_{g}^2\bigr) \,,
    \\
    A^0_j &= 
    b^2 K  \frac{X_j}{\mu}\left[ -1
    - \frac{K^{-1}}{(\psi+c_2)} z^k X_k  \frac{\nu}{\tau} \left( - 2+ \frac{\nu}{\tau\mu} (-N^2 + |X|_g^2)\right)\right]
    + \Ord\bigl(| z|_{g}^2\bigr), 
    \\
    A^0_{ij}& = b^2 \left[ g_{ij}  + X_i X_j \frac{\nu}{\tau \mu} \Big( -2 + (-N^2 + |X|_g^2)\frac{\nu}{\tau\mu}\Big) \right]
    \\&\qquad
    + 2 b^2 z_{(i} X_{j)}\left[- \frac{b}{\tau \mu} + (-N^2 + |X|_g^2)\frac{b\nu}{\tau\mu} - 2\frac{K}{\mu}\right] + \Ord\bigl(| z|_{g}^2\bigr) ,
\\
      A^k_0 &= \Ord\bigl(| z|_{\gc}^2\bigr),
      \qquad A^k_j =b^2 K \delta^k_j + \Ord\bigl(| z|_{\gc}^2\bigr), \\
    A^k_{ij} &=b^2 \left[ \Bigl(\frac{2}{\psi+c_2}\Bigr)g_{ij} z^k + \frac{\kappa K}{(\psi+c_2)} (\delta^k_j z_i +\delta^k_i z_j) \right] 
    \\&\qquad
    + b^3 X_i X_j z^k\frac{\nu}{\tau\mu}\left[ -2 + (-N^2 + |X|^2_g)\frac{\nu}{\tau\mu}\right] + \Ord\bigl(| z|_{g}^2\bigr)\,. 
\end{align*}

Premultiply the PDE system \eqref{conf-Eul-F} by  $(A_0^0)^{-1}$ and rewrite it as
$$
(A_0^0)^{-1} A^0\del{\tau}Z + \frac{1}{\tau} \frac{\tau}{\nu} (A_0^0)^{-1}A^k \nabla_{k} Z =(A_0^0)^{-1} Q^{\tr}(H-B^0Y) \,.
$$

We define
\begin{align*}
M^0(Z) &:= (A^0_0)^{-1} A^0(Z),
\quad
C^k := \left( (A^0_0)^{-1} A^k \right)|_{z^i=0},
\\
M^k(Z) & := \frac\tau\nu(A^0_0)^{-1} A^k(Z) - C^k \AND
(\mathcal{F}_0 \,, \mathcal{F}_j )^{\tr} \define \tau (A_0^0)^{-1} Q^{\tr}(H-B^0Y)\,.
\end{align*}
Then, thanks to Lemma \ref{lem:transf-Eul-general}, the following PDE holds
$$
M^0 \del{\tau} Z +\frac1\tau (C^k + M^k) \nabla_{k}Z = \frac{1}{\tau} (\mathcal{F}_0, \mathcal{F}_j)^{\tr}.$$

Noting that $(A_0^0)^{-1} =  (b^2 K)^{-1}  +  \Ord\bigl(|X|^2_g + | z|_{g}^2\bigr)$
we compute
\begin{align}
\begin{split}\label{M0}
M^0(Z) &
= \begin{pmatrix}
1 & 0 \\ 0 & K^{-1} g_{ij}
\end{pmatrix} + \begin{pmatrix} 0 & X_j \mu^{-1} \\ X_j \mu^{-1}  & \Ord(|X|^2|z|_g\mu^{-1}) \end{pmatrix} + \Ord\bigl(|X|^2_g + | z|_{g}^2\bigr) ,
\\&= \begin{pmatrix}
1 & 0 \\ 0 & K^{-1}g_{ij}
\end{pmatrix} + |\tau| e^{-  T} \begin{pmatrix} 0 & 1 \\ 1 & 0 \end{pmatrix} + \Ord\bigl(|X|^2_g + | z|_{g}^2\bigr) ,
\end{split}
\end{align}
and
\begin{align}\notag
M^k(Z) &= \begin{pmatrix} 0 & 0 \\
0 & \frac1{K}\Bigl(\frac{2}{\psi+c_2}\Bigr)g_{ij} z^k + \frac{\kappa}{(\psi+c_2)} (\delta^k_j z_i +\delta^k_i z_j)
    \end{pmatrix} + \begin{pmatrix} 0 & 0 \\0  & \Ord(|X|^2|z|_g\mu^{-1}) \end{pmatrix} +  \Ord\bigl(|X|^2_g + | z|_{g}^2\bigr) 
    \\
    &= \begin{pmatrix} 0 & 0 \\
0 & \Ord(|z|_g)
    \end{pmatrix} +  \Ord\bigl(|X|^2_g + | z|_{g}^2\bigr) .
    \label{Mk}
\end{align} 

Next, we introduce the following notation
\begin{align*}
R_i \define \frac{1-3K}{\mu \nu} w_i \partial_\tau \Psi +  m_i + \frac{K}{ \mu \nu} w_i D_2 a \partial_\tau g_{lk} \cdot  z^l  z^k.
\end{align*}
Using Lemma \eqref{lem:rough-sources} we compute
\begin{align*}
|\mathcal{F}_0| 
	&\lesssim | \tau (bs)^{-2} \left[K D_1a ( \ell - D_2 a \partial_\tau g_{ij} \cdot  z^i z^j) + b' z^j R_j\right]| + \Ord(|X|_g^2 + |z|_g^2)
\\
	&\lesssim |\psi+c_2| \left( |\tau \ell| + |\partial_T g| |z|_g^2\right) + |z|_g|X||\nu| + |z|_g|\tau m_j|_g + \Ord(|X|_g^2 + |z|_g^2)
	\\
	&\lesssim
	\Ord(|X|_g^2 + |z|_g^2+|X|_g+|\nabla X|_g+|\hat{N}|+|\nabla N|_g+|\Sigma|_g+|\partial_{T}X|_g).
\end{align*}
Similarly, for each $j$ we have
\begin{align*}
\mathcal{F}_j &= \tau (bs)^{-2}\left[2K D_2a z_j(\ell-D_2a\partial_\tau g_{ij} \cdot z^i  z^j)  + b R_j\right]+\Ord(|X|_g^2 + |z|_g^2)
\\&\lesssim
(3-K^{-1})z_j + (3-K^{-1})|\nu| |X|_g + |\psi+c_2||\tau m_i|_g + 
|z|_g|\tau \ell| + \Ord\bigl(|X|_g^2 + | z|_{g}^2\bigr)
\\&\lesssim
(3-K^{-1})z_j + \Ord(|X|_g^2 + |z|_g^2+|X|_g+|\nabla X|_g+|\hat{N}|+|\nabla N|_g+|\Sigma|_g+|\partial_{T}X|_g).
\end{align*}
Defining  $F = \mathcal{F} - \Bc \Pbb Z $ we see that the above bounds on the components of $\mathcal{F}$ imply the estimate on $F$ in \eqref{estF}. 
\end{proof}

\subsection{Definition and coercivity properties of the fluid energy}

Now we are equipped to introduce the higher order energy functionals, which are coercive with respect to the Sobolev-energies of their respective order.

\begin{Def}\label{def:energy}
	Let $ 0\leq s \leq k-1 $. We define the following energy functionals 
	\begin{align*}
		E_{s}(Z)&=\frac{1}{2}\sum_{l\leq s}\int_{M}\langle \nabla^{l}Z,M^{0}\nabla^{l}Z\rangle_g \,\mu_g , \quad
		\E_{s}(Z)=\frac{1}{2}\sum_{l\leq s}\int_{M}\langle \Pbb \nabla^{l}Z,M^{0}\Pbb\nabla^{l}Z\rangle_g \,\mu_g,  \\
		\dot{E}_{s}(Z)&=\int_{M}\langle \nabla^{s}Z,M^{0}\nabla^{s}Z\rangle_g \,\mu_g .
	\end{align*}
	We refer to $ \E_{s} $ as the \emph{parallel energy} and $ \dot{E}_{s} $ as the \emph{homogeneous energy}, of order $s$ respectively. 
\end{Def} 

\begin{lem}[Equivalence of the energy norms]
	Under the bootstrap assumptions \eqref{bootstrap}, for $ 0\leq s\leq k-1 $, we have that
	\begin{equation*}
		E_{s}(Z)\cong \|Z\|_{H^{s}}^{2}, \quad \E_{s}(Z)\cong \|\Pbb Z\|_{H^{s}}^{2}, \quad \dot{E}_{s}(Z)\cong \|Z\|_{\dot{H}_{s}}. 
	\end{equation*}
\end{lem}

\begin{proof}
	From \eqref{M0} we see that $ M^{0}$ is approximately the diagonal matrix $\text{diag}(1,g) $. Together with Sobolev embedding, this yields 
	\begin{equation*}
		|\dot{E}_{s}(Z)-\|Z\|_{\dot{H}^{s}}^{2}|\lesssim  \|Z\|_{\dot{H}^{s}}^{2}+(e^{-(1+\lambda)T}+\||X|_{g}^{2}+|z|_{g}^{2}\|_{L^{\infty}})\|Z\|_{\dot{H}^{s}}^{2}\lesssim \epsilon. 
	\end{equation*}
	The other expressions involve similar estimates, and so we find the energy norms are equivalent to their Sobolev counterparts. 
\end{proof}

\subsection{Fluid energy estimates of lower order}

Before we consider the time evolution of the lowest order energy $ E(Z) $ we first derive some preliminary estimates and identities for the coefficient matrices.

\begin{lem}[Properties of the coefficient matrices]\label{est:matrices}
We have that
	\begin{align*}
			|\Pbb^{\perp} (\partial_{T}M^{0}-\nabla_{a}M^{a})\Pbb|_{\op}+|\Pbb (\partial_{T}M^{0}-\nabla_{a}M^{a})\Pbb^{\perp}|_{\op}	\quad & \\
		+|\Pbb (\partial_{T}M^{0}-\nabla_{a}M^{a})\Pbb|_{\op}&\lesssim |\Pbb Z|+|\Pbb \nabla Z| + \epsilon e^{-\lambda T},\\
		|\Pbb^{\perp} (\partial_{T}M^{0}-\nabla_{a}M^{a})\Pbb^{\perp}|_{\op}&\lesssim |\Pbb Z|^{2}+|\Pbb \nabla Z|^{2} + \epsilon e^{-\lambda T}.
	\end{align*}
	Furthermore, for any square matrix $ A$, we have the following identity
	\begin{equation*}
		\langle Z, AZ\rangle=\langle \Pbb Z, (\Pbb A \Pbb)\Pbb Z \rangle+\langle \Pbb^{\perp} Z, (\Pbb^{\perp} A \Pbb)\Pbb Z \rangle+\langle \Pbb Z, (\Pbb A \Pbb^{\perp})\Pbb^{\perp} Z \rangle+\langle \Pbb^{\perp} Z, (\Pbb^{\perp} A \Pbb^{\perp})\Pbb^{\perp}Z \rangle.
	\end{equation*}
\end{lem}
\begin{proof}
	First note that, using the equations of motion \eqref{Milne:fluid-PDE}, $  \partial_{T}M^{0} $ can be schematically rewritten using the chain rule as 
	\begin{equation*}
		\partial_{T}M^{0}(Z,g,X,N)=D_{Z}M^{0}\cdot\partial_{T}Z+D_{g}M^{0}\cdot\partial_{T}g+D_{X}M^{0}\cdot\partial_{T}X+D_{N}M^{0}\partial_{T}N.
	\end{equation*}
	The last three terms involving the geometry can be estimated by 
	\begin{equation*}
		\|D_{g}M^{0}\partial_{T}g+D_{X}M^{0}\partial_{T}X+D_{N}M^{0}\partial_{T}N\|_{L^{\infty}}\lesssim \|\partial_{T}g\|_{H^{2}}+\|\partial_{T}X\|_{H^{2}}+\|\partial_{T}N\|_{H^{2}}\lesssim \epsilon e^{-\lambda T},
	\end{equation*}
	using Sobolev-embedding as well as the bootstrap assumptions and \eqref{estgtime}. 
	
	Next, we inspect
	\begin{equation*}
			D_{Z}M^{0}\cdot\partial_{T}Z=D_{Z}M^{0}\cdot(M^{0})^{-1}\left((C^{k}+M^{k})\nabla_{k}Z-\Bc \Pbb Z -F\right).
	\end{equation*}
	Every term in this expression is a sum of matrices of the form $ f(Z, \nabla Z)D_{Z}M^{0} $, where the function $ f $ schematically indicates everything coming from $ (M^{0})^{-1}\left((C^{k}+M^{k})\nabla_{k}Z-\Bc \Pbb Z -F\right) $. Now from   \eqref{M0} we see that to leading order $ M^{0} $ is given by the diagonal matrix $ \text{diag}(1,K^{-1}g_{ij}) $. Thus, these leading order terms vanish under the derivative $ D_{Z} $. The remaining terms in $D_{Z}M^0$ involve $X,\nabla X$ or $ |z|_{g}$ or $ |z|_{g}^{2} $, where we used the fact that $ D_{z^{i}}(|z|_{g}^{2})=\Ord(|z|_{g}) $. Hence, the first inequality follows for $ \partial_{T}M^{0} $. As explained in Remark \ref{M0component}, we again find that $ \Pbb^{\perp}D_{Z}M^{0}\Pbb^{\perp}=\Ord(|X|_{g}) $ and hence the second inequality of the statement is trivially satisfied for $ \partial_{T}M^{0} $. 
	
	Inspecting $ M^{k} $ using \eqref{Mk}, we see the leading order term is the diagonal matrix $ \text{diag}(0, \Ord(|z|_{g})) $. We immediately infer that  each piece of $\nabla_kM^k$, except for $ \Pbb \nabla_{k}M^{k} \Pbb $, may be estimated in the same way as the error terms of $ \partial_{T}M^{0} $. For the $ \Pbb \nabla_{k}M^{k} \Pbb $ part,  the leading order term consists of terms of order of $ |z|_{g}$ as well as $|\nabla z|_{g} $. These are estimated by $|\Pbb Z|$ and $|\Pbb \nabla Z|$.
	
	Finally, using $\Pbb^2 =\Pbb$, $\Pbb^T = \Pbb$ and $\id = \Pbb+\Pbb^\perp$, we calculate
	\begin{align*}
		\langle Z, AZ \rangle &= \langle (\Pbb +\Pbb^{\perp})Z, A (\Pbb+\Pbb^{\perp})Z \rangle\\
		&=\langle \Pbb\Pbb Z, A \Pbb\Pbb Z \rangle+\langle \Pbb^{\perp}\Pbb^{\perp} Z, A \Pbb\Pbb Z \rangle+\langle \Pbb\Pbb Z, A \Pbb^{\perp}\Pbb^{\perp} Z \rangle+\langle \Pbb^{\perp}\Pbb^{\perp} Z, A \Pbb ^{\perp}\Pbb^{\perp}Z \rangle\\
		&=\langle \Pbb Z, (\Pbb A \Pbb)\Pbb Z \rangle+\langle \Pbb^{\perp} Z, (\Pbb^{\perp} A \Pbb)\Pbb Z \rangle+\langle \Pbb Z, (\Pbb A \Pbb^{\perp})\Pbb^{\perp} Z \rangle+\langle \Pbb^{\perp} Z, (\Pbb^{\perp} A \Pbb^{\perp})\Pbb^{\perp}Z \rangle.
	\end{align*}
\end{proof}
\begin{comment}
\begin{align*}
\partial_T E(Z)
	&= \int_M \langle Z, M^0 \partial_T Z\rangle \mu_g + \frac12 \int_M \langle Z, (\partial_T M^0)Z\rangle \mu_g
\\
	&\leq
	\int_M \langle Z, (C^k+M^k) \nabla_k Z\rangle - \int_M \langle \Pbb Z, \Bc \Pbb Z\rangle - \int_M \langle \Pbb Z, \Pbb F\rangle 
	- \int_M \langle Z, \Pbb^\perp F\rangle 
	 + \frac12 |\partial_T M^0| \| Z \|_{L^2(M)}^2
\\
	&\leq 
	- (3-s^{-2}) E_p(Z)
	+ C \left( |(\nabla_k M^k)| + |\partial_T M^0| \right) E(Z)
	+ E_p(Z)^{1/2} \| \Pbb F\|_{L^2(M)}
	+ E(Z)^{1/2} \| \Pbb^\perp F\|_{L^2(M)}
\\
	&\leq 
	- (3-s^{-2}) E_p(Z)
	+ C \left( |(\nabla_k M^k)| + |\partial_T M^0| \right) E(Z)
	+ \varepsilon e^{-\lambda T} E(Z)^{1/2} 
	+ E(Z)^{3/2}
	+ ...
\end{align*}
\end{comment}

Equipped with the preliminary results above we can now establish an estimate for the time evolution of the energy of order zero.

\begin{prop}\label{est:zeroenergy}
	Under \eqref{bootstrap} there exists a constant $C>0$ such that 
	\begin{equation*}
		\partial_{T}E_{0}(Z)\leq (3-K^{-1}+C\epsilon)\E_{0}(Z)+C\epsilon e^{-\lambda T}+C\epsilon \E_{1}(Z). 
	\end{equation*}
\end{prop}
\begin{rem}
	The parameter $ K<\frac{1}{3} $ is fixed, and so $ 3-K^{-1}+C\epsilon $ is negative provided $ \epsilon $ is sufficiently small. This overall negative sign of the parallel part of the energy  is essential to closing the final energy estimates. 
\end{rem}

\begin{proof}
	A computation using the equations of motion (\ref{Milne:fluid-PDE}) and Lemma \ref{timederivative} yields
	\begin{align*}
		\partial_{T}E_{0}(Z)&\leq  \int_{M}\langle Z,M^{0}\partial_{T}Z\rangle +\frac{1}{2} \int_{M}\langle Z,(\partial_{T}M^{0})Z\rangle+C\epsilon e^{-\lambda T}E_{0}(Z)\\
		&=\int_{M}\langle Z,(C^{a}+M^{a})\nabla_{a}Z\rangle -\int_{M}\langle \Pbb Z,\Bc \Pbb Z\rangle -\int_{M}\langle  Z, F \rangle  \\
		&\qquad +\frac{1}{2}\int_{M}\langle Z,(\partial_{T}M^{0})Z\rangle+C\epsilon e^{-\lambda T}E_{0}(Z).
	\end{align*}
	Using integration by parts and the fact that the matrices $ C^{a} $ and $ M^{a} $ are symmetric we find 
	\begin{equation*}
		\int_{M}\langle Z,(C^{a}+M^{a})\nabla_{a}Z\rangle=-\frac{1}{2}\int_{M}\langle Z,\nabla_{a}(C^{a}+M^{a})Z\rangle=-\frac{1}{2}\int_{M}\langle Z,(\nabla_{a}M^{a})Z\rangle.
	\end{equation*}
	Note that $ a $ is indeed an index that stems from a spatial vector field, in particular $ z $, which allows for integration by parts. Hence we have
	\begin{equation*}
		\partial_{T}E_{0}(Z)
		\leq (3-K^{-1})\E_{0}(Z)
		+ \int_{M}\langle Z,(\partial_{T}M^{0}-\nabla_{a} M^{a})Z\rangle-\int_{M}\langle  Z, F \rangle
		+C\epsilon e^{-\lambda T}E_{0}(Z).
	\end{equation*}
	
	We start with the term involving $ F $. Using equation (\ref{estF}), we obtain
	\begin{align*}
		|\int_{M}\langle Z,F\rangle|\lesssim \|Z\|_{L^{2}}\|F\|_{L^{2}}\lesssim \|Z\|_{L^{2}}\|F\|_{L^{\infty}}\lesssim  \epsilon^{2} e^{-\lambda T}.
	\end{align*}
	We next analyse the term involving $ \partial_{T}M^{0}-\nabla_{k}M^{k} $. Using the matrix identity in Lemma \ref{est:matrices} as well as Hölder's inequality we find
	\begin{align*}
		\int_{M}&\langle Z, (\partial_{T}M^{0}-\nabla_{a}M^{a})Z\rangle \\&=\int_{M}\langle \Pbb Z, \Pbb (\partial_{T}M^{0}-\nabla_{a}M^{a}) \Pbb\Pbb Z\rangle +\int_{M}\langle \Pbb^{\perp} Z, \Pbb^{\perp} (\partial_{T}M^{0}-\nabla_{a}M^{a})\Pbb\Pbb Z\rangle \\
		&\qquad +\int_{M}\langle \Pbb Z, \Pbb (\partial_{T}M^{0}-\nabla_{a}M^{a})\Pbb^{\perp}\Pbb^{\perp} Z\rangle+\langle \Pbb^{\perp} Z, \Pbb^{\perp} (\partial_{T}M^{0}-\nabla_{a}M^{a})\Pbb^{\perp}\Pbb^{\perp} Z\rangle \\
		&\lesssim \E_{0}(Z)|\Pbb (\partial_{T}M^{0}-\nabla_{a}M^{a}) \Pbb|_{\op}+\|Z\|_{L^{\infty}}^{2} \int_{M}|\Pbb^{\perp}(\partial_{T}M^{0}-\nabla_{a}M^{a}) \Pbb^{\perp}|_{\op}\\
		&\qquad +\|Z\|_{L^{\infty}}\int_{M}\big(|\Pbb^{\perp} (\partial_{T}M^{0}-\nabla_{a}M^{a}) \Pbb|_{\op}+|\Pbb (\partial_{T}M^{0}-\nabla_{a}M^{a}) \Pbb^{\perp}|_{\op}\big)|\Pbb Z|\\
		&\lesssim \E_{0}(Z)(|\Pbb Z|+|\Pbb \nabla Z|+\epsilon e^{-\lambda T})+\|Z\|_{L^{\infty}}^{2}(\E_{0}(Z)+\E_{1}(Z)+\epsilon e^{-\lambda T})\\
		&\qquad+\|Z\|_{L^{\infty}}\|\Pbb Z\|_{L^{\infty}}(\E_{0}(Z)+\E_{1}(Z)+\epsilon e^{-\lambda T}).
	\end{align*}
	In conclusion, using Sobolev embedding and the smallness assumption, we find that
	\begin{align*}
		\int_{M}\langle Z, (\partial_{T}M^{0}-\nabla_{a}M^a)Z\rangle 
		&\lesssim (\sqrt{\epsilon}+\epsilon e^{-\lambda T}) \E_{0}(Z)+\epsilon (\E_{0}(Z)+\E_{1}(Z))+\epsilon e^{-\lambda T},
	\end{align*}
	where we used Young's inequality and Cauchy-Schwarz in the last step. 
	All together we get
	\begin{equation*}
		\partial_{T}E_{0}(Z)\leq (3-K^{-1}+C \epsilon)\E_{0}(Z)+C \epsilon e^{-\lambda T}+C\epsilon \E_{1}(Z). 
	\end{equation*}
\end{proof}
\begin{rem}
	Note that, due to the inclusion of $ \E_{1}(Z) $, the energy estimate in Proposition \ref{est:zeroenergy} does not close. However, this is not an issue, since this problem is unique to the first order and this term can be absorbed into a negative definite term of the type $ -c\E_{1}(Z) $, $ c>0 $,  appearing in the final estimate. 
\end{rem}

\section{Estimates of the higher order fluid energy}

In this section we derive an estimate for the higher-order fluid energies. These estimates will be weaker compared to the zero order case. However, in the end this will be remedied by exploiting the lower order estimate. 
We start be deriving an expression for the time-evolution of the homogeneous part of the energy of order $ \ell\geq 1$. 

\begin{lem}
Let $1 \leq \ell \leq k-1$. Under \eqref{bootstrap} there is a constant $C>0$ such that
\begin{equation}\label{higherordercalc}
	\begin{aligned}
	\frac{1}{2}\partial_{T}\Big( \int_{M}\langle \nabla^{\ell}Z,M^{0}\nabla^{\ell}Z\rangle\Big) &\leq (3-K^{-1}) \int_M \langle \nabla^\ell Z, \Pbb \nabla^\ell Z\rangle 
	+ \frac12 \int_{M}\langle \nabla^{\ell}Z,(\partial_T M^0 -\nabla_{a}M^{a})\nabla^{\ell}Z \rangle
	\\&\qquad
	+ \int_M \langle \nabla^\ell Z, G^{\ell}\rangle 
	+ 	C\epsilon e^{-\lambda T}E_{\ell}(Z),
	\end{aligned}
\end{equation}
where
\begin{align*}
G^\ell &\define -\Bc [\nabla^{\ell},\Bc^{-1} M^{0}](M^{0})^{-1}\big((C^{a}+M^{a})\nabla_{a}Z-\Bc \Pbb Z-F\big)\\
	&\qquad + \Bc [\nabla^{\ell},\Bc^{-1}(C^{a}+M^{a})]\nabla_{a}Z-\Bc \nabla^{\ell}(\Bc^{-1} F) -M^{0}[\nabla^{\ell},\partial_{T}]Z-(C^{a}+M^{a})[\nabla^{\ell},\nabla_{a}]Z.
\end{align*}
\end{lem}

\begin{proof}
Let $ \ell\geq 1 $. 
Applying the operator $ \Bc \nabla^{\ell} \Bc^{-1} $ to the fluid equations of motion of \eqref{Milne:fluid-PDE}  yields
\begin{align}\label{eomhigherorder}
	M^{0}\partial_{T}\nabla^{\ell}Z-(C^{a}+M^{a})\nabla_{a}\nabla^{\ell}Z&=-\Bc \Pbb \nabla^{\ell}Z+G^{\ell}.
\end{align}
The error term $ G^{\ell}$ is computed as
\begin{align*}
	G^{\ell}&= -\Bc [\nabla^{\ell},\Bc^{-1} M^{0}]\partial_{T}Z+\Bc [\nabla^{\ell},\Bc^{-1}(C^{a}+M^{a})]\nabla_{a}Z-\Bc \nabla^{\ell}(\Bc^{-1} F)\\
	&\qquad -M^{0}[\nabla^{\ell},\partial_{T}]Z-(C^{a}+M^{a})[\nabla^{\ell},\nabla_{a}]Z\\
	&=-\Bc [\nabla^{\ell},\Bc^{-1} M^{0}](M^{0})^{-1}\big((C^{a}+M^{a})\nabla_{a}Z-\Bc \Pbb Z-F\big)\\
	&\qquad + \Bc [\nabla^{\ell},\Bc^{-1}(C^{a}+M^{a})]\nabla_{a}Z-\Bc \nabla^{\ell}(\Bc^{-1} F) -M^{0}[\nabla^{\ell},\partial_{T}]Z-(C^{a}+M^{a})[\nabla^{\ell},\nabla_{a}]Z,
\end{align*}
where we simply inserted the equations of motion \eqref{Milne:fluid-PDE} in the last step. 

Using Lemma \ref{timederivative} and \eqref{eomhigherorder}, we have
\begin{equation*}
	\begin{aligned}
	\frac12 \partial_{T}\int_{M}\langle \nabla^{\ell}Z,M^{0}\nabla^{\ell}Z\rangle&\leq  \int_{M}\langle \nabla^{\ell}Z,M^{0}\partial_{T}\nabla^{\ell}Z\rangle+\frac{1}{2}\int_{M}\langle \nabla^{\ell}Z,(\partial_{T}M^{0})\nabla^{\ell}Z\rangle+C\epsilon e^{-\lambda T}E(Z)\\
	&=\frac12 \int_{M}\langle \nabla^{\ell}Z,(\partial_T M^0 -\nabla_{a}M^{a})\nabla^{\ell}Z\rangle
	-\int_M \langle \nabla^\ell Z, \Bc \Pbb \nabla^{\ell}Z\rangle
	\\&\qquad+ \int_M \langle \nabla^\ell Z, G^{\ell}\rangle
	+C\epsilon e^{-\lambda T}E(Z).
	\end{aligned}
\end{equation*}
\end{proof}

We now proceed by introducing some preliminary estimates. 

\begin{lem}\label{auxestmatrices}
	The fluid coefficient matrices and the inhomogeneity $ F $ obey the following estimates
	\begin{align}
		|\nabla\Pbb M^{0}\Pbb^{\perp}|_{\op}+ |\nabla\Pbb^{\perp} M^{0}\Pbb|_{\op}&\lesssim e^{-T} \Ord(|X|_{g})+\Ord(|X|_{g}^{2}+|z|_{g}^{2}+|\nabla z|_{g}^{2}),\label{M0conj}\\
		|\Pbb (M^{0})^{-1}\Pbb^{\perp}|_{\op}+ |\Pbb^{\perp} (M^{0})^{-1}\Pbb|_{\op}&\lesssim e^{-T} \Ord(|X|_{g})+\Ord(|X|_{g}^{2}+|z|_{g}^{2}),\notag\\
		|\Pbb^{\perp}(M^{a}\nabla_{a}Z-\Bc \Pbb Z-F)\Pbb|+&\notag\\
		|\Pbb^{\perp}(M^{a}\nabla_{a}Z-\Bc \Pbb Z-F)\Pbb^{\perp}|&\lesssim \epsilon e^{-\lambda T}+ \Ord(|X|_{g}^{2}+|z|_{g}^{2}+|\nabla z|_{g}^{2}),\notag\\
		| M^{a}\nabla_{a}Z-\Bc \Pbb Z-F|&\lesssim  \epsilon e^{-\lambda T}+\Ord(|z|_{g}+|X|_{g}^{2}),\notag\\
		|\Pbb^{\perp} C^{a}\nabla_{a}Z|&= |\Pbb^{\perp} C^{a} \Pbb \nabla_{a}Z|\lesssim \Ord(|\nabla z|_{g}).\notag
	\end{align}
\end{lem}
\begin{proof}
	Recall from Remark \ref{rem:conjeffect} that conjugating matrices with $ \Pbb $ and $ \Pbb^{\perp} $ picks out specific matrix elements. Considering the explicit form of $ M^{0} $ in \eqref{M0} we realize that the leading order term is given by $ \text{diag}(1,K^{-1}g) $ which is annihilated by the covariant derivative. Hence follows \eqref{M0conj}. 
	
	Inspecting \eqref{Mk}, we see that the leading order term in $ M^{a} $ is of order $ \Ord(|z|_{g}) $ and is picked out by the conjugation $ \Pbb M^{a}\Pbb $. In conjunction with \eqref{estF} we conclude that the statements about $ M^{a}\nabla_{a}Z-\Bc Z-F $ above are true. Note that we treated the $ \nabla Z $ terms as negligible using Sobolev embedding. The rest of the statements follows in a similar fashion. 
\end{proof}

\begin{lem}\label{auxhigherorder}
	Let $1 \leq \ell\leq k-1 $. Under \eqref{bootstrap}, there is a constant $C>0$ such that the following estimates hold
	\begin{equation}\label{inhomhigh}\begin{split}
		\int_{M}\langle \nabla^{\ell}Z,(\partial_{T}M^{0}-\nabla_{a}M^{a})\nabla^{\ell}Z\rangle &\leq C\epsilon \dot{\E}_{\ell}(Z)+C\epsilon \E_{k-1}(Z),\\
		\int_{M}\langle \nabla^{\ell}Z,G^{\ell}\rangle &\leq C\epsilon e^{\lambda T}E_{k-1}(Z)^{\frac{1}{2}}+C\epsilon \E_{k-1}(Z)+\bad^{\ell},
    \end{split}
	\end{equation}
	where the problematic error term $ \bad^{\ell} $ is defined as
	\begin{align*}
		\bad^{\ell}&\define \int_{M}\langle \Pbb \nabla^{\ell}Z,\Pbb C^{a}\Pbb^{\perp}[\nabla^{\ell},\nabla_{a}]Z\rangle+\langle \Pbb^{\perp} \nabla^{\ell}Z,\Pbb^{\perp} C^{a}\Pbb[\nabla^{\ell},\nabla_{a}]Z\rangle.
	\end{align*}
	
\end{lem}

\begin{rem}
	The error terms $ \bad^{\ell} $  involve the matrix $C^a$, defined in \eqref{Milne-Ckeq}. The off-diagonal structure of $C^a$ causes a `mixing' effect in $ \bad^{\ell} $  by leading to terms of the type $ \Ord(\psi \cdot |z|_{g}) $. Such terms cannot be controlled in terms of the energy $ \E $ since $\psi$ is not controlled by the parallel energy. Furthermore, note that these terms arise  due to the curved geometry and are not present in a setting in which derivatives commute. Hence, these terms require  additional treatment given below in the form of modified energies. 
\end{rem}

\begin{proof}
	The proof of the divergence estimate is essentially the same as in lowest order. However, we repeat the calculation and add some details. Using the identity from Lemma \ref{est:matrices}
	\begin{align*}
		\int_{M}&\langle \nabla^{\ell}Z, (\partial_{T}M^{0}-\nabla_{a}M^{a})\nabla^{\ell}Z\rangle \\
		&=\int_{M}\langle \Pbb \nabla^{\ell}Z, \Pbb (\partial_{T}M^{0}-\nabla_{a}M^{a}) \Pbb\Pbb\nabla^{\ell} Z\rangle  +\int_{M}\langle \Pbb^{\perp}\nabla^{\ell} Z, \Pbb^{\perp} (\partial_{T}M^{0}-\nabla_{a}M^{a})\Pbb\Pbb \nabla^{\ell}Z\rangle \\
		&\qquad +\int_{M}\langle \Pbb \nabla^{\ell}Z, \Pbb (\partial_{T}M^{0}-\nabla_{a}M^{a})\Pbb^{\perp}\Pbb^{\perp}\nabla^{\ell} Z\rangle +\int_{M}\langle \Pbb^{\perp} \nabla^{\ell}Z, \Pbb^{\perp} (\partial_{T}M^{0}-\nabla_{a}M^{a})\Pbb^{\perp}\Pbb^{\perp} \nabla^{\ell}Z\rangle\\
		&\lesssim  \dot{\E}_{\ell}(Z)\|\Pbb (\partial_{T}M^{0}-\nabla_{a}M^{a}) \Pbb\|_{L^{\infty}}
		+\dot{E}_{\ell}(Z) \|\Pbb^{\perp} (\partial_{T}M^{0}-\nabla_{a}M^{a}) \Pbb^{\perp}\|_{L^{\infty}}\\
		&\qquad +\int_{M}\big(|\Pbb^{\perp} (\partial_{T}M^{0}-\nabla_{a}M^{a}) \Pbb|_{\op} +|\Pbb (\partial_{T}M^{0}-\nabla_{a}M^{a}) \Pbb^{\perp}|_{\op}\big)|\Pbb \nabla^{\ell} Z| |\nabla^{\ell}Z|.
	\end{align*}
	Performing a similar estimate as in the proof of Proposition \ref{est:zeroenergy} we find
	\begin{equation*}
		\int_{M}\langle \nabla^{\ell}Z, (\partial_{T}M^{0}-\nabla_{a}M^{a})\nabla^{\ell}Z\rangle\lesssim \epsilon \dot{\E}_{\ell}(Z)+\epsilon \E_{k-1}(Z)+\epsilon e^{-\lambda T}.
	\end{equation*}

	Deriving \eqref{inhomhigh} is slightly more involved. We start be establishing a statement for a general tensor-valued vector $ V $ (i.e.~the components of $ V $ are tensors fields on $ M $):
	\begin{equation}\label{200}\begin{split}
		\int_{M}&\langle \nabla^{\ell}Z, \Bc [\nabla^{\ell},\Bc^{-1}M^{0}](M^{0})^{-1}V\rangle
		\\
		&=\int_{M}\langle \Pbb\nabla^{\ell}Z, \Bc [\nabla^{\ell},\Bc^{-1}M^{0}](M^{0})^{-1} \Pbb V\rangle+\int_{M}\langle  \Pbb^{\perp}\nabla^{\ell}Z, \Bc [\nabla^{\ell},\Bc^{-1}M^{0}](M^{0})^{-1} \Pbb ^{\perp}V\rangle\\
		&\qquad+\int_{M}\langle  \Pbb\nabla^{\ell}Z, \Bc [\nabla^{\ell},\Bc^{-1}\Pbb M^{0}]( \Pbb+ \Pbb^{\perp})(M^{0})^{-1} \Pbb^{\perp}V\rangle\\
		&\qquad +\int_{M}\langle  \Pbb^{\perp}\nabla^{\ell}Z, \Bc [\nabla^{\ell},\Bc^{-1}\Pbb^{\perp}M^{0}]( \Pbb+ \Pbb^{\perp})(M^{0})^{-1} \Pbb V\rangle.
	\end{split}\end{equation}
	The first line on the right hand side of \eqref{200} can be estimated by
	\begin{equation*}
		\|\Pbb\nabla^{\ell}Z\|_{L^{2}}\|\nabla M^{0}\|_{H^{k-2}}\|\Pbb V\|_{H^{k-2}}+\|\nabla^{\ell}Z\|_{L^{2}}\|\nabla M^{0}\|_{H^{k-2}}\|\Pbb^{\perp} V\|_{H^{k-2}},
	\end{equation*}
	where we used standard estimates for the commutator (see e.g.~\cite[Theorem A.3]{beyeroliynyk}). The second and third terms on the right hand side of \eqref{200} can be expanded as
	\begin{align*}
		&\int_{M}\langle  \Pbb\nabla^{\ell}Z, \Bc \big([\nabla^{\ell},\Bc^{-1}\Pbb M^{0}\Pbb^{\perp}](M^{0})^{-1} \Pbb^{\perp}V+[\nabla^{\ell},\Bc^{-1}\Pbb M^{0}]\Pbb(M^{0})^{-1} \Pbb^{\perp}V\big)\rangle\\
		&\qquad+\int_{M}\langle  \Pbb^{\perp}\nabla^{\ell}Z, \Bc \big([\nabla^{\ell},\Bc^{-1}\Pbb^{\perp} M^{0}\Pbb](M^{0})^{-1} \Pbb V+[\nabla^{\ell},\Bc^{-1}\Pbb^{\perp} M^{0}]\Pbb^{\perp}(M^{0})^{-1} \Pbb  V\big)\rangle\\
		&\lesssim\|\Pbb \nabla^{\ell}Z\|_{L^{2}}\left(\|\nabla \Pbb M^{0}\Pbb^{\perp}\|_{H^{k-2}}\|(M^{0})^{-1}\Pbb^{\perp }V\|_{H^{k-2}}+\|\nabla\Pbb M^{0}\|_{H^{k-2}}\|\Pbb (M^{0})^{-1}\Pbb^{\perp}\|_{H^{k-2}}\|V\|_{H^{k-2}}\right)\\
		&\qquad +\| \nabla^{\ell}Z\|_{L^{2}}\left(\|\nabla \Pbb^{\perp} M^{0}\Pbb\|_{H^{k-2}}\|(M^{0})^{-1}\Pbb V\|_{H^{k-2}}+\|\nabla\Pbb^{\perp} M^{0}\|_{H^{k-2}}\|\Pbb^{\perp} (M^{0})^{-1}\Pbb\|_{H^{k-2}}\|V\|_{H^{k-2}}\right).
	\end{align*}
	Utilizing these estimates for $ V=(M^{0})^{-1}\big((M^{a}+C^{a})\nabla_{a}Z-\Bc\Pbb Z -F\big) $, in conjunction with Lemma \ref{auxestmatrices}, we conclude that
	\begin{equation*}
		\int_{M}\langle \nabla^{\ell}Z, \Bc [\nabla^{\ell},\Bc^{-1}M^{0}](M^{0})^{-1}\big((M^{a}+C^{a})\nabla_{a}Z-\Bc\Pbb Z -F\big)\rangle\lesssim \epsilon \E_{k-1} + \epsilon e^{-\lambda T} E_{k-1}(Z)^{\frac{1}{2}}.
	\end{equation*}
	Now we tackle the rest of the inhomogeneity described in $ G^{\ell} $. 
	\begin{align*}
		\int_{M}&\langle \nabla^{\ell}Z, \Bc [\nabla^{\ell},\Bc^{-1}(C^{a}+M^{a})]\nabla_{a}Z\rangle\\
		&=\int_{M}\langle \nabla^{\ell}Z, \Bc [\nabla^{\ell},\Bc^{-1}C^{a}]\nabla_{a}Z\rangle+\int_{M}\langle \nabla^{\ell}Z, \Bc [\nabla^{\ell},\Bc^{-1}M^{a}]\nabla_{a}Z\rangle\\
		&=\int_{M}\langle \Pbb \nabla^{\ell}Z, \Bc [\nabla^{\ell},\Bc^{-1}\Pbb M^{a}(\Pbb+\Pbb^{\perp})]\nabla_{a}Z\rangle+\int_{M}\langle\Pbb^{\perp} \nabla^{\ell}Z, \Bc [\nabla^{\ell},\Bc^{-1}\Pbb^{\perp}M^{a}(\Pbb+\Pbb^{\perp})]\nabla_{a}Z\rangle\\
		&\lesssim \|\Pbb \nabla^{\ell}Z\|_{L^{2}}\|\nabla \Pbb Z\|_{H^{k-2}}\|\nabla Z\|_{H^{k-2}}+\|\nabla (\Pbb Z)\|_{H^{k-2}}^{2}\|\nabla Z\|_{H^{k-2}}+\|X\|_{H^{k-2}}^{2}\|Z\|_{H^{k-2}}^{2},
	\end{align*}
	where we used that $ M^{a} $ only has leading order terms in the $ \Pbb M^{a}\Pbb $ component, as seen in \eqref{Mk}, and the last term stems from the error terms in $ M^{a} $. Furthermore, using \eqref{estF}, we have that
	\begin{align*}
		\int_{M}\langle \nabla^{\ell}Z,\Bc \nabla^{\ell}(\Bc^{-1}F)\rangle\lesssim \|Z\|_{H^{k-1}}\left(\|\Pbb Z\|_{H^{k-1}}^{2}+\|X\|_{H^{k-1}}^{2}+\epsilon e^{- T}\right).
	\end{align*}
	Lastly, using $\Pbb C^{a} \Pbb= \Pbb^{\perp} C^{a} \Pbb^{\perp}=0$, we find that
	\begin{align*}
		\int_{M}\langle \nabla^{\ell}Z,(C^{a}+M^{a})[\nabla^{\ell},\nabla_{a}]Z\rangle
		&=\int_{M}\langle (\Pbb+\Pbb^{\perp}) \nabla^{\ell}Z,(C^{a}+M^{a})(\Pbb+\Pbb^{\perp})[\nabla^{\ell},\nabla_{a}]Z\rangle\\
		&=\int_{M}\langle \Pbb \nabla^{\ell}Z,\Pbb M^{a}\Pbb[\nabla^{\ell},\nabla_{a}]Z\rangle+\|X\|_{H^{k-1}}^{2}\|Z\|_{H^{k-1}}^{2}+\bad^{\ell}
		\\& 
		\leq C\varepsilon\|\Pbb Z\|_{H^{k-1}}^{2} +C \|X\|_{H^{k-1}}^{2}\|Z\|_{H^{k-1}}^{2}+\bad^{\ell}.
	\end{align*}
	Note that we estimated the non-leading order terms of $ M^{a} $ by $ C \|X\|_{H^{k-1}}^{2}\|Z\|_{H^{k-1}}^{2} $. Applying the bootstrap assumptions, we see that these terms are exponentially decaying. 
	The last term to control in $G^\ell$ is the due to the geometry:
	\begin{align*}
		\int_{M}\langle \nabla^{\ell}Z, M^{0} [\nabla^{\ell},\partial_{T}]Z\rangle&\lesssim \|\partial_{T}\Gamma\|_{H^{k-1}}E_{k-1}(Z).
	\end{align*}
	The norm of the time derivative of the Christoffel symbols can simply by estimated by $ \epsilon e^{-\lambda T} $ by applying the equations of motion of the geometry in \eqref{eq:EoM-pT-g-Sigma}.
\end{proof}

\subsection{Corrected Fluid Energies}\label{corr-en}

We now introduce a corrected energy which allows us to compensate for the problematic  terms $ \bad^{\ell} $. We start by giving explicit expressions for the first and second order corrected energies and show how these remedy the problem. Then, we present a process to extend these concepts to higher orders. 

\begin{Def}[First and second order corrected energies] \label{corrected12}
	Define corrected first order energies  by
	\begin{align*}
		\Eco_{1}(Z) &\define E_{1}(Z)-\frac{1}{9}\int_{M} \langle \Pbb Z, M^{0} \Pbb Z\rangle,\qquad
		\dot{\Eco}_{1}(Z) \define \dot{E}_{1}(Z)-\frac{1}{9}\int_{M} \langle \Pbb Z, M^{0} \Pbb Z\rangle.
	\end{align*}
	We define the corrected second order energies $\Eco_{2}$ and $\dot{\Eco}_{2}$ by 
	\begin{align*}
				\Eco_{2}(Z) &\define E_{2}-\frac{4}{9}\int_{M} \langle \Pbb \nabla Z,M^{0} \Pbb\nabla Z\rangle-\frac{4}{81}\int_{M} \langle \Pbb Z, M^{0} \Pbb Z\rangle,
		\\
		\dot{\Eco}_{2}(Z) &\define \dot{E}_{2}-\frac{4}{9}\int_{M} \langle \Pbb \nabla Z,M^{0} \Pbb\nabla Z\rangle-\frac{4}{81}\int_{M} \langle \Pbb Z, M^{0} \Pbb Z\rangle.
	\end{align*}
\end{Def}

\begin{rem}
	Note that while the coefficients of the correction terms are negative, their modulus is less then $1/2$. Hence, the functionals $ \Eco $ are equivalent to the functions $ E $ and thus also to the Sobolev-norms of their respective order.
\end{rem}

With these new energies at our disposal we are now able to formulate improved estimates.

\begin{prop}\label{higherorderestimate}
	For $ \ell=1,2 $ we have 
	\begin{equation*}
		\begin{aligned}
			\partial_{T}\Eco_{\ell}(Z)\leq  (3-K^{-1}+C\epsilon)\E_{\ell}(Z)+C\epsilon e^{-\lambda T}+ C\epsilon \E_{k-1}(Z).
		\end{aligned}
	\end{equation*}
\end{prop}

\begin{proof}
First, note that by using the explicit form of $C^a$, the critical term $B^\ell$ can be written as
		\begin{align*}
			B^\ell 
			&=\int_{M}\nabla^{\ell}z^{a}[\nabla^{\ell},\nabla_{a}]\psi+\nabla^{\ell}\psi [\nabla^{\ell},\nabla_{a}]z^{a}.
		\end{align*}

We start by considering the first order case $\ell=1$. Using Lemma \ref{auxhigherorder} as well as  (\ref{higherordercalc}), we find
		\begin{equation*}
			\begin{aligned}
				\partial_{T}\dot{\Eco}_{1} \leq (3-K^{-1}+C\epsilon)\dot{\E}_{1}(Z)+C\epsilon e^{-\lambda T}+ C\epsilon \E_{k-1}(Z)+\bad^{1}-\frac{2}{9}\int_{M} \langle \Pbb Z, M^{0}\partial_{T} \Pbb Z\rangle.
			\end{aligned}
		\end{equation*}
	 	Note that we used that $ |\partial_{T}M^{0}|\lesssim \epsilon e^{-\lambda T}+\E_{k-1} $. 
	
		When $\ell=1 $, we obtain 
		\begin{align}\label{22a}
			B^1 &= \int_{M}g^{mn} \nabla_{m}\psi [\nabla_{n},\nabla_{a}]z^{a}=\int_{M}g^{mn} \nabla_{m}\psi \Riem[g]^{a}{}_{bna}z^{b}
			=
			-\int_{M}g^{mn} \nabla_{m}\psi \Ric[g]_{bn}z^{b} .
		\end{align}
		We first consider \eqref{22a}, when all the geometric quantities are with respect to $\gamma$ instead of $g$:
		\begin{align*}
		- \int_{M}\gamma^{mn} \partial_{m}\psi \Ric[\gamma]_{nb}z^{b} 
		&= 
		\frac29 \int_{M} \gamma^{mn} \; \gamma_{bn}z^{b}\nabla_{m}\psi
		= \frac{2}{9}\int_{M}z^{m}\nabla_{m}\psi .
		\end{align*}
		So we find that the error term is
		\begin{align*}
		&\int_{M} g^{mn} (\Ric[g]_{bn}-\Ric[\gamma]_{bn})z^{b}\partial_{m}\psi\mu_{g} 		+ \int_{M} (g-\gamma)^{mn} \Ric[\gamma]_{bn}z^{b}\partial_{m}\psi\mu_{g}
		\lesssim
		\|g-\gamma\|_{H^{2}}\int_{M}|z\nabla\psi|
            \lesssim \epsilon e^{-T},
		\end{align*}
		where we used Lemma \ref{Riemannestimate}. Employing  (\ref{Milne:fluid-PDE}), we further calculate
		\begin{align*}
			-\frac{2}{9}\int_{M} \langle \Pbb Z, \partial_{T} \Pbb Z\rangle&=-\frac{2}{9}\int_{M} \langle \Pbb Z, \Pbb\big((C^{a}+M^{a})\nabla_{a}Z-\Bc \Pbb Z-F\big)\rangle\\
			&\leq \underbrace{(K^{-1}-3)\frac{2}{9}\int_{M}\langle \Pbb Z, \Pbb Z \rangle}_{\eqqcolon(*)}-\frac{2}{9}\int_{M}z^{m}\nabla_{m}\psi +\epsilon \Ord(|z|_{g}^{2}+e^{-\lambda T}).
		\end{align*}
		Note that the positive term $ (*) $ from $ \Bc $ is not a problem, since this Lemma gives us the negative term $ (3-K^{-1}+\epsilon) \E(Z) $ in the estimate for $ E_{0}(Z) $ and it is easy to see that
		\begin{equation*}
			(3-K^{-1}+C\epsilon) \E(Z)+(K^{-1}-3)\frac{2}{9}\int_{M}\langle \Pbb Z, \Pbb Z \rangle\leq C(3-K^{-1}+C\epsilon) \E(Z),
		\end{equation*}
		where $ C>0 $.
		Next, we compute the problematic term in the case $ \ell=2 $:
		\begin{align*}
			B^2 = \int_{M}g^{mn}g^{rs}& \nabla_{m}\nabla_{r}\psi [\nabla_{n}\nabla_{s},\nabla_{a}]z^{a}+g^{mn}g^{rs} \nabla_{m}\nabla_{r}z^{a} [\nabla_{n}\nabla_{s},\nabla_{a}]\psi\\
			&=\int_{M}g^{mn}g^{rs} \nabla_{m}\nabla_{r}\psi (\nabla_{n}[\nabla_{s},\nabla_{a}]z^{a}+[\nabla_{n},\nabla_{a}]\nabla_{s}z^{a})+g^{mn}g^{rs} \nabla_{m}\nabla_{r}z^{a} [\nabla_{n},\nabla_{a}]\nabla_{s}\psi\\
			&=\int_{M}g^{mn}g^{rs} \nabla_{m}\nabla_{r}\psi \big(\nabla_{n}(\Riem^{a}{}_{bsa}z^{b})+\Riem^{a}{}_{bna}\nabla_{s}z^{b}-\Riem^{b}{}_{sna}\nabla_{b}z^{a}\big)\\
			&\qquad-\int_{M}g^{mn}g^{rs} \nabla_{m}\nabla_{r}z^{a} \Riem^{b}{}_{sna}\nabla_{b}\psi.
		\end{align*}
		We proceed by replacing the Riemann tensor with the one on the background and estimating the result by $ \|g-\gamma\|_{H^{k-1}} $ as before. Furthermore, by changing the measure to $ \mu_{\gamma} $ and the connection to $ \hat{\nabla} $, we pick up exponentially decaying error terms as in the case of $ l=1 $. We continue the calculation on the background:
		\begin{align*}
			\int_{M}\gamma^{mn}\gamma^{rs}& \hat{\nabla}_{m}\hat{\nabla}_{r}\psi
			\big(\frac{2}{9}\gamma_{bs}\hat{\nabla}_{n}z^{b}+\frac{2}{9}\gamma_{bn}\hat{\nabla}_{s}z^{b}-\frac{1}{9}( \gamma^{b}{}_{a}\gamma_{sn}-\gamma^{b}{}_{n}\gamma_{sa})\hat{\nabla}_{b}z^{a}\big)\\
			&\qquad-\frac{1}{9}\int_{M}\gamma^{mn}\gamma^{rs}\hat{\nabla}_{m}\hat{\nabla}_{s}z^{a}(\gamma^{b}{}_{a}\gamma_{sn}-\gamma^{b}{}_{n}\gamma_{sa})\hat{\nabla}_{b}\psi\\
			&=\frac{5}{9}\int_{M}\gamma^{mn}\hat{\nabla}_{m}\hat{\nabla}_{b}\psi\hat{\nabla}_{n}z^{b}-\frac{1}{9}\int_{M}\gamma^{mn}\hat{\nabla}_{m}\hat{\nabla}_{n}\psi\hat{\nabla}_{b}z^{b}-\frac{1}{9}\int_{M}\gamma^{mn}\hat{\nabla}_{m}\hat{\nabla}_{n}z^{b}\hat{\nabla}_{b}\psi
			\\&\quad
			+\frac{1}{9}\int_{M}\gamma^{mn}\hat{\nabla}_{m}\hat{\nabla}_{a}z^{a}\hat{\nabla}_{n}\psi=\frac{6}{9}\int_{M}\gamma^{mn}\hat{\nabla}_{m}\hat{\nabla}_{b}\psi\hat{\nabla}_{n}z^{b}+\frac{2}{9}\int_{M}\gamma^{mn}z^{b}\hat{\nabla}_{b}\hat{\nabla}_{m}\hat{\nabla}_{n}\psi \\
			&=\frac{6}{9}\int_{M}\gamma^{mn}z^{b}\hat{\nabla}_{m}\hat{\nabla}_{b}\psi\hat{\nabla}_{n}+\frac{2}{9}\int_{M}\gamma^{mn}z^{b}\big(\hat{\nabla}_{m}\hat{\nabla}_{b}+[\hat{\nabla}_{b},\hat{\nabla}_{m}]\big)\hat{\nabla}_{n}\psi \\
			&=\frac{4}{9}\int_{M}\gamma^{mn}\hat{\nabla}_{m}\hat{\nabla}_{b}\psi\hat{\nabla}_{n}z^{b}-\frac{2}{9}\int_{M}\gamma^{mn}z^{b}\Riem[\gamma]^{a}{}_{nbm}\hat{\nabla}_{a}\psi \\
			&=\frac{4}{9}\int_{M}\gamma^{mn}\hat{\nabla}_{m}\hat{\nabla}_{b}\psi\hat{\nabla}_{n}z^{b}+\frac{4}{81}\int_{M}z^{b}\hat{\nabla}_{b}\psi.
	\end{align*}
	We also see that, utilizing Lemma \ref{auxhigherorder},
	\begin{align*}
		-\frac{4}{9}&\int_{M}g^{mn}\langle \Pbb \nabla_{m}Z,M^{0}\partial_{T}\Pbb \nabla_{n}Z\rangle =-\frac{4}{9}\int_{M}g^{mn}\langle \Pbb \nabla_{m}Z,\Pbb \left((C^{a}+M^{a})\nabla_{a}\nabla_{n}Z-\Bc \Pbb \nabla_{n}Z+G_{m}\right)\rangle  \\
		&\lesssim (K^{-1}-3)\frac{4}{9}\int_{M}g^{mn}\langle \Pbb \nabla_{m}Z,\Pbb \nabla_{n}Z \rangle -\frac{4}{9}\int_{M}g^{mn}\nabla_{m}z^{b}\nabla_{n}\nabla_{b}\psi\\
		&\qquad-\frac{4}{9}\int_{M}g^{mn}\underbrace{\langle \Pbb \nabla_{m}Z, \Pbb C^{a}[\nabla_{a},\nabla_{n}]Z\rangle}_{=0}+C \epsilon e^{-\lambda T},
	\end{align*}
	where, in the last line, we used that $  \Pbb C^{a}[\nabla_{a},\nabla_{n}]Z=(0,[\nabla_{i},\nabla_{n}]\psi)^{\tr}=0 $. In the same spirit as in the case of $ l=1 $ we can conclude that the term introduced by $ \Bc $ does not affect the overall estimate as it can be absorbed in the negative definite term $ (3-K^{-1}+C\epsilon)\E_{0} $.  Summing our results for $ \dot{\Eco}_{1} $ and $ \dot{\Eco}_{2} $ over $ l\leq 1 $ and $ l\leq 2 $ respectively, together with Lemma \ref{est:zeroenergy}, yields the desired result. 
\end{proof}

\begin{prop}[Higher order corrections]\label{finalenergy}
	Corrected energies $ \Eco_{\ell} $ for all orders $0\leq \ell\leq k-1 $ that are equivalent to the energies in Definition \ref{def:energy} can be constructed and satisfy the estimate in Proposition \ref{higherorderestimate}, i.e.
	\begin{equation*}
		\begin{aligned}
			\partial_{T}\Eco_{\ell}(Z)\lesssim  (3-K^{-1}+\epsilon)\E_{\ell}(Z)+\epsilon e^{-\lambda T}+ \epsilon \E_{k-1}(Z).
		\end{aligned}
	\end{equation*}
\end{prop}

\begin{rem}
	Constructing these energies, while in principle straightforward, is a tedious process and does not add in a meaningful way to the arguments. Hence, we  omit an explicit construction and instead give a recipe how to construct them and prove why they satisfy the above estimate. 
\end{rem}

\begin{proof}
	The ideas presented in the proof of Proposition \ref{higherorderestimate} can be generalized to higher orders in the following fashion. Schematically, at order $ \ell $, the error $ B $ takes the form 
	\begin{equation*}
		|\int_{M}\nabla^{\ell}z^{a}[\nabla^{\ell},\nabla_{a}]\psi+\nabla^{\ell}\psi [\nabla^{\ell},\nabla_{a}]z^{a}|\lesssim |\int_{M}\nabla^{\ell}z\ast\nabla^{\ell-1}\psi +\nabla^{\ell}\psi\ast \nabla^{\ell-1}z|\lesssim |\int_{M}\nabla^{\ell}\psi\ast\nabla^{\ell-1}z|,
	\end{equation*}
	where we integrated by parts in the last step. By $ \ast $ we denote an arbitrary contraction $ T_{1}\ast T_{2} $ with the metric $ g $. Furthermore, note that $ z $ here is always contracted with one of the covariant derivatives, so these tensors are indeed of the same valence. In the expression above there are terms of the form $ \nabla\Riem $. We tacitly assumed these to be negligible, since we can always replace any geometric quantities like $ \Riem $ or $ \nabla $ by their background counterparts and $ \hat{\nabla}\Riem[\gamma]=0 $ by the Ricci decomposition as done in Lemma \ref{Riemannestimate}. 
	
	Now the terms in this $ \ast $-product may involve contractions over multiple derivatives of $ z $ or $ \psi $, i.e.~$ \gamma^{mn}\nabla_{m}\nabla_{n}z $ or divergence-like terms, i.e.~$ \nabla_{m}\nabla_{n}z^{m} $. All the terms are of order $ \ell $ and $ \ell-1 $. In these cases we can always interchange derivatives, in which case the curvature produces terms of order e.g.~$ \nabla^{\ell-2}\psi\nabla^{\ell-1}z $, which can then be rearranged using integration by parts. After all these manipulations, one ends up with an error of the type
	\begin{align*}
		\int_{M}c_{\ell\ell}\nabla^{\ell-1}\nabla_{a}\psi\nabla^{\ell-1}z^{a}+c_{\ell\ell-1}\nabla^{\ell-2}\nabla_{a}\psi\nabla^{\ell-2}z+\dots+c_{\ell1}z^{a}\nabla_{a}\psi .
	\end{align*}
	To illustrate this idea in practice, we note the explicit constants $ c_{22}= \frac{4}{9}$ and $ c_{21}=\frac{4}{81} $ in the case of $ \ell=2 $ are derived explicitly in the proof of Proposition \ref{higherorderestimate}.
	
	In the worst case one of these constants, w.l.o.g.~lets say $ c_{\ell \ell} $, is large and positive. This forces us to add a negative definite term to our energy. However, we may put a small coefficient in front of the top order term in our energy, i.e.~
	\begin{align*}
		\delta_{\ell}\int_{M}\langle \nabla^{\ell}Z,M^{0}\nabla^{\ell}Z\rangle,
	\end{align*}
	so that the coefficient appears as $ c_{\ell \ell}\delta_{\ell}\ll1 $. Thus, the correction term needed is
	\begin{align*}
		-c_{\ell \ell}\delta_{\ell}\int_{M}\langle \Pbb \nabla^{\ell-1}Z, \Pbb M^{0}\nabla^{\ell-1}Z\rangle,
	\end{align*}
	which conserves the positive definiteness and coercivity of the energy. After deriving this term and employing the equations of motion, another potential problem is the positive definite term
	\begin{align*}
		c_{\ell \ell}\delta_{\ell}\int_{M}\langle \Pbb \nabla^{\ell-1}Z, \Bc \Pbb \nabla^{\ell-1}Z\rangle. 
	\end{align*}
	However, this does not pose a threat for closing the energy estimates, since $ \delta_{\ell} $ is fixed and as small as necessary, and hence in the end the term can be absorbed by the negative definite term $ (3-K^{-1})\E_{\ell-1} $, as $ 3-K^{-1}<0 $ is fixed. Another potential problem is posed by the term generating the error term $ \bad $ initially. We end up with a term of the form 
	\begin{equation*}
		-c_{\ell \ell}\delta_{\ell}\int_{M}\langle \Pbb \nabla^{\ell-1}Z,\Pbb C^{a}[\nabla^{\ell-1},\nabla_{a}]Z\rangle=-c_{\ell \ell}\delta_{\ell}\int_{M}\nabla^{\ell-1}z^{i}[\nabla^{\ell-1},\nabla_{i}]\psi.
	\end{equation*}
	However, we see that this again straightforwardly falls into our sum 
	\begin{align*}
		\int_{M}\tilde{c}_{\ell\ell-1}\nabla^{\ell-1}\psi\nabla^{\ell-2}z+\dots+\tilde{c}_{\ell 1}\nabla\psi z.
	\end{align*}
	with new constants $ \tilde{c}_{\ell i} $. The same process may now be repeated until the lowest order is reached. In this case the term generated by the matrix $ C $ simply vanishes. 
\end{proof}

\section{Geometric energy}\label{EinEulBR-3}

\begin{lem}[Matter source terms]\label{mattersources}
In the case of a relativistic fluid stress energy tensor \eqref{EnMom}, the rescaled matter quantities take the form
	\begin{align*}
	E&=\rho_{0}e^{a(1+K)}e^{-3KT}\left((1+K)(\hat{v}^{\tau})^{2}N^{2}+K\right), \qquad j^{a}=\rho_{0}e^{a(1+K)}e^{(1-3K)T}N(1+K)\hat{v}^{\tau}v^{a},\\
	\eta&=\tilde{E}+\rho_{0}e^{a(1+K)}e^{-3KT}\left((1+K)(|X|_{g}^{2}(\hat{v}^{\tau})^{2}-2 g_{ab}X^{a}v^{b}v^{\tau}+|v|_{g}^{2})+3K\right),\\
	S_{ab}&=\rho_{0}e^{a(1+K)}e^{(-1-3K)T} \big((1+K)\big(X_{a}X_{b}(\hat{v}^{\tau})^{2}- X_{a}\hat{v}^{\tau}v_{b}-X_{b}\hat{v}^{\tau}v_{a}+v_{a}v_{b}\big)+K(2+K)\big).
	\end{align*}
\end{lem}

\begin{proof}
Using Definition \ref{matterquant}, a straightforward computation shows that
	\begin{align*}
	\tilde{E}&=\tilde{T}^{\mu\nu}n_{\mu}n_{\nu}=\tilde{\rho}\left((1+K)(\bar{v}^{\tau})^{2}\tilde{N}^{2}+K\bar{g}^{00}\tilde{N}^{2}\right),\\
%	&=\tilde{\rho}\left((1+K)(\hat{v}^{\tau})^{2}N^{2}+K\right),\\
	\tilde{j}_{a}&=\tilde{N}\tilde{T}^{0\mu}\bar{g}_{\mu a}=\tilde{N}\left(\tilde{\rho}(1+K)\bar{v}^{\tau}\bar{v}^{\mu}+K\tilde{\rho}\bar{g}^{0\mu}\right)\bar{g}_{a\mu},\\
%	&=(-\tau)^{-1}N\tilde{\rho}(1+K)\hat{v}^{\tau}v_{a},\\
	\tilde{\eta}&=\tilde{E}+\tilde{g}^{ab}\tilde{T}_{ab}=\tilde{g}^{ab}\bar{g}_{a\mu}\bar{g}_{b\nu}\left((1+K)\tilde{\rho}\bar{v}^{\mu}\bar{v}^{\nu}+K\tilde{\rho}\bar{g}^{\mu\nu}\right)\\
	&=\tilde{E}+(1+K)\tilde{\rho}\tau^{-2}\tilde{g}^{ab}\left(X_{a}X_{b}(\hat{v}^{\tau})^{2}- X_{a}\hat{v}^{}v_{b}- X_{b}\hat{v}^{\tau}v_{a}+v_{a}v_{b}\right)+3K\tilde{\rho},\\ 
	\tilde{S}_{ab}&=\tilde{T}_{ab}-\frac{1}{2}\tr_{\bar{g}}\tilde{T}\cdot \tilde{g}_{ab}=\tilde{T}_{ab}-\frac{1}{2}\bar{g}_{\mu\nu}\tilde{T}^{\mu\nu}\cdot\tilde{g}_{ab}\\
	&=K\tilde{\rho}\tau^{-2}g_{ab}+(1+K)\tilde{\rho}\tau^{-2}\left(X_{a}X_{b}(\hat{v}^{\tau})^{2}- X_{a}\hat{v}^{\tau}v_{b}- X_{b}\hat{v}^{\tau}v_{a}+v_{a}v_{b}\right)\\
	&\qquad-2\tau^{-2}K\tilde{\rho}g_{ab}+2\tilde{\rho}(1+K)\tau^{-2}g_{ab}.
	\end{align*}
	Furthermore, inverting $ \zeta $ using \eqref{zeta:dynamic}, we find 
	\begin{align*}
	\tilde{\rho} = \rho_{0}e^{\zeta(1+K)}(-\tau)^{3(1+K)}.
	\end{align*}
	Using this relation, we end up with the desired results. Note that we used $ \tilde{j}^{a}=\tilde{g}^{ab}\tilde{j}_{a}=(-\tau)^{2}g^{ab}\tilde{j}_{ab} $.
\end{proof}

\begin{lem}[Estimates on the matter variables]\label{matterest}
	The matter components obey the following estimates
	\begin{align*}
	|\tau|\|\eta\|_{H^{k-1}}+|\tau|^{2}\|j\|_{H^{k-1}}+|\tau|\|\partial_{T}\eta\|_{H^{k-2}}+|\tau|\|\partial_{T}j\|_{H^{k-2}}&\lesssim \epsilon e^{-(1+3K)T},\\
	|\tau|\|S\|_{H^{k-1}}&\lesssim \epsilon e^{-(2+3K)T}.
	\end{align*}
\end{lem}

\begin{proof}
Using the expressions for the matter components given in Lemma \ref{mattersources}, a straightforward application of Sobolev estimates leads to the following:
	\begin{equation*}
		|\tau|\|\eta\|_{H^{k-1}}\lesssim |\tau|^{1+3K}\|e^{a(\psi,|z|_{g})}(\hat{v})^{2}N^{2}\|_{H^{k-1}}\lesssim |\tau|^{1+3K}\|Z\|_{H^{k-1}}\|(\hat{v})^{2}\|_{H^{k-1}}\|N^{2}\|_{H^{k-1}}\lesssim \epsilon e^{-(1+3K)T}.
	\end{equation*}
	The calculation for $ j $ and $ S $ is analogous. However, the estimate of the time derivatives require a more careful analysis. We have the following bounds on the individual bounds
	\begin{align*}
		\|\partial_{T}(e^{a(\psi,|z|_{g})})\|_{H^{k-2}}&\lesssim \|D_{1}a\partial_{T}\psi\|_{H^{k-2}}+\|D_{2}a\partial_{T}|z|_{g}^{2}\|_{H^{k-2}}\lesssim  \epsilon +\epsilon e^{-\lambda T},\\
		\|\partial_{T}\hat{v}\|_{H^{k-2}}&\lesssim \|\partial_{T}N\|_{H^{k-2}}+\|\partial_{T}X\|_{H^{k-2}}+\|\partial_{T}\Pbb Z\|_{H^{k-2}}\lesssim \epsilon+\epsilon e^{-\lambda T},\\
		\|\partial_{T}g\|_{H^{k-2}}&\lesssim \epsilon e^{-\lambda T}.
	\end{align*}
	These rough estimates straightforwardly imply the estimates of the time derivatives. 
\end{proof}

\subsection{Geometric energy}
	Since $ \mathscr{L}_{g,\gamma} $ is an elliptic operator on the compact space $M$, it has a discrete spectrum of eigenvalues. On account of a result from \cite{KK-15}, the smallest eigenvalue $ \lambda_{0} $ of the operator $ \mathscr{L}_{g,\gamma} $ satisfies $ \lambda_{0}\geq \tfrac{1}{9} $ and the operator has trivial kernel. 

\begin{Def}[Geometric energy]\label{def:geomtricenergy}
 We define the constant $ \alpha(\lambda_{0},\delta_{\alpha}) $ and $ c_{E}$ as
	\begin{equation*}
		\alpha \coloneqq \begin{cases}
			1, &\lambda_{0}<\frac{1}{9},\\
			1-\delta_{\alpha}, &\lambda_{0}=\frac{1}{9},
		\end{cases} \AND c_{E}\coloneqq \begin{cases}
			1, &\lambda_{0}<\frac{1}{9},\\
			9(\lambda_{0}-\xi), &\lambda_{0}=\frac{1}{9},
		\end{cases}
	\end{equation*}
	where $ \delta_{\alpha}=\sqrt{1-9(\lambda_{0}-\xi)} $ with $ 0<\xi<1 $, which remains to be fixed. Given these constants, we recall from \cite{AnderssonMoncrief:2011, AF20} the geometric energy $ \mathcal{E}_{m} $ and the correction term $ \Gamma_{m} $ of order $ m\geq 1 $ as
	\begin{equation*}
		\begin{aligned}
			E^{(g)}_{m}&=\frac{1}{2}\int_{M}\langle 6\Sigma, \mathcal{L}_{g,\gamma}^{m-1}6\Sigma\rangle +\frac{9}{2}\int_{M}\langle (g-\gamma),\mathcal{L}_{g,\gamma}^{m}(g-\gamma)\rangle,\qquad \Gamma^{(g)}_{m}&=\int_{M}\langle 6\Sigma, \mathcal{L}_{g,\gamma}^{m-1}(g-\gamma)\rangle. 
		\end{aligned}
	\end{equation*}
	Finally we define the corrected geometric energy $ E_{s} $ of order $ s\geq 1 $ as
	\begin{equation*}
		\mathcal{E}_{s}\define \sum_{1\leq m\leq s}E^{(g)}_{m}+c_{E}\Gamma^{(g)}_{m}
	\end{equation*}
\end{Def}

\section{Proof of global existence and stability}
\begin{proof}[Proof of Theorem \ref{thm:Milne_stability}]
	The smallness condition of the initial data in (\ref{initialsmallness}) implies the existence of a constant $ C_{0} $ such that
	\begin{equation*}
		\|N-3\|_{H^{k}}+\|X\|_{H^{k}}+\|\partial_{T}N\|_{H^{k-1}}+\|\partial_{T}X\|_{H^{k-1}}+\Eco_{k-1}(Z)+\mathcal{E}_{k}\leq C_{0}\epsilon.
	\end{equation*}
	Employing the inequality of Proposition \ref{finalenergy}  yields
	\begin{equation*}
		\partial_{T}\Eco_{k-1}(Z)\leq (3-K^{-1}+C\epsilon){\E}_{k-1}(Z)+C\epsilon e^{-\lambda T}.
	\end{equation*}
	Hence, if $ \epsilon $ is sufficiently small, and using that $3-K^{-1}+C\epsilon <0$, we have that
	\begin{equation*}
		\|\rho\|_{H^{k-1}}+\|u\|_{H^{k-1}}\lesssim \Eco_{k-1}(Z)(T)\leq  C\epsilon\int_{T_{0}}^{T}e^{-\lambda s}ds\leq C_{0}\epsilon,
	\end{equation*}
	where we potentially redefined $ T_{0} $. 
	
Next, we improve the bootstrap on the lapse and shift using an estimate given in \cite[Prop. 17]{AF20} together with our matter estimates from Lemma \ref{matterest}:
	\begin{align*}
	\|N-3\|_{H^{k}} + \|X\|_{H^{k}}&\lesssim \|\Sigma\|_{H^{k-2}}^{2}+|\tau|\|\eta\|_{H^{k-2}}+|\tau|^{2}\|Nj\|_{H^{k-2}}+\|g-\gamma\|_{H^{k-1}}^{2}
	\\&\lesssim
	\epsilon^{2}e^{-2\lambda T}+\epsilon e^{-(1+3K)T},
	\end{align*}
	which closes the bootstrap assumptions on the lapse and shift made in \eqref{bootstrap}. 
Similarly, to improve the bootstrap on $\partial_T N, \partial_T X$ we use \cite[Lem. 18]{AF20}  (see also the correction in \cite[Prop. 7.2]{BarzegarFajman20pub}) and the matter estimates of  Lemma \ref{matterest} to get
	\begin{align*}
	\|\partial_{T}N\|_{H^{k-1}}&\lesssim\|\hat{N}\|_{H^{k-1}}+\|X\|_{H^{k-1}}+\|\Sigma\|_{H^{\ell-1}}^{2}+\|g-\gamma\|_{H^{k-1}}^{2}
	+|\tau|\|S\|_{H^{k-3}}+|\tau|\|\eta\|_{H^{k-3}}
		\\&\quad
		+|\tau|\|\partial_{T}\eta\|_{H^{k-3}}
	\\&\lesssim
	\epsilon^2 e^{-2\lambda T} + \epsilon e^{-(1+3K)T},
\intertext{and}
	\|\partial_{T}X\|_{H^{k-1}}&\lesssim\|\hat{N}\|_{H^{k-3}}+\|\partial_{T}\hat{N}\|_{H^{k-2}}+\|X\|_{H^{k-1}}+\|\Sigma\|_{H^{k-2}}^{2}+\|g-\gamma\|_{H^{k-2}}^{2}+|\tau|^{2}\|j\|_{H^{k-3}}\\
	&\qquad+|\tau|^{2}\|\partial_{T}j\|_{H^{k-3}}+|\tau|\|\partial_{T}\eta\|_{H^{k-3}}
	\\&\lesssim
		\epsilon^2 e^{-2\lambda T} + \epsilon e^{-(1+3K)T} .
	\end{align*}
	
The evolution of the geometric energy is given by
		\begin{align*}
		\partial_{T}\mathcal{E}_{k}&\leq -2\alpha \mathcal{E}_{k}+6\mathcal{E}_{k}^{\frac{1}{2}}\|NS\|_{H^{k-1}}+C\mathcal{E}_{k}^{\frac{3}{2}}+C\mathcal{E}_{k}^{\frac{1}{2}}\big(|\tau|\|\eta\|_{H^{k-1}}+|\tau|^{2}\|Nj\|_{H^{k-2}}\big)
		\\&\leq
		-2\alpha\mathcal{E}_{k}+6C\mathcal{E}_{k}^{\frac{1}{2}}\epsilon^{2}e^{-(1+3K)T}+C\mathcal{E}_{s}^{\frac{3}{2}}.
		\end{align*}
The first line above is from \cite[Lem. 20]{AF20} and the second uses our matter estimates. 
	This energy estimate can be closed in a similar way as in \cite{AF20}. This amounts to rescaling the energy by an exponential factor depending on $ \delta_{\alpha} $ as introduced in Definition \ref{def:geomtricenergy}. For further details of this construction see \cite[\textsection 9]{AF20}. Also note that the decay in this estimate is stronger due to the additional factor of $ e^{-KT} $. 
	Thanks to \cite[Lem. 19]{AF20}, we have the coercivity
		\begin{equation*}
		\|g-\gamma\|_{H^{k}}+\|\Sigma\|_{H^{k-1}}\lesssim \mathcal{E}_{k}.
	\end{equation*}
 Finally, future completeness relies on the rate of decay of the perturbation of the unrescaled geometry and matter fields. The exact details are very similar to that given in \cite{AF20} and the conclusion uses the completeness criterion given in \cite{CB_Cotsakis}.
\end{proof}

\begin{rem}\label{arbitrarygeometry}
	Note that all the arguments in Section \ref{EinEulBR} and following are not specific to a closed manifold close to the \emph{Milne-geometry} but, except for a few details, carry over in a straightforward manner to a \emph{fixed spatial background geometry}. The main motivation to consider constant negative Einstein-curvature is to guarantee stability of the geometry. 
	However, when considering an arbitrary non-dynamic spatial background, we have to adapt our strategy when considering the energy corrections. As noted before, these stem from the non-commutativity of the spatial derivatives. Since we lack the luxury of an explicit background we cannot employ the same idea presented in Lemma \ref{corrected12}, i.e.~replace $ \Ric[g] $ by $ \Ric[\gamma] $ and exploit the fact that their difference is small. 
	However, since our spatial manifold $ (M,g) $ is fixed and closed, we know that
	\begin{equation*}
		0 \leq |\Ric[g]|_{\op}\leq \alpha
	\end{equation*}
	for some $ \alpha>0 $. Remember that the first order error term is proportional to 
	\begin{equation*}
		\delta_{1}\int_{M}\Ric[g]_{mn}z^{n}\nabla^{m}\psi,
	\end{equation*}
	where $ \delta_{1}>0 $ is a possibly small constant that we set before the top order part of the norm to scale this error term. We may then correct our energy as
	\begin{equation*}
		\tilde{E}_{1}(Z)\coloneqq E_{1}(Z)-\frac{\delta_{1}}{2}\int_{M}\langle\Pbb Z,\Ric[g]M^{0}\Pbb Z\rangle,
	\end{equation*}
	which is still equivalent, given $ \delta_{1} $ is sufficiently small. The time derivative of $ M^{0} $ and $ \Ric[g] $ as well as additional terms from Lemma \ref{timederivative} are  negligible. Employing the equations of motion similar to the proof of Lemma \ref{higherorderestimate}, in addition to the compensating term we also receive the additional 
	\begin{equation*}
		\delta_{1}\int_{M}\langle\Pbb Z,\Ric[g]\Bc\Pbb Z\rangle\leq \alpha\delta_{1}\E_{0}(Z).
	\end{equation*}
	Given that $ \delta_{1} $ is sufficiently small, this term gets absorbed into the negative definite term $ -c\E_{0}(Z) $.
	
	Note that we can always decompose the curvature tensor into $ \Ric[g] $ and $ g $. Hence, we can, in a similar fashion as in Proposition \ref{finalenergy}, put a small constant $ \delta_{l} $ in front of the top order $ n $ derivatives in the energy and compensate for the terms, by adding and appropriate term of order $ n-1 $ depending only on the fluid velocity. The result then follows as a corollary of Theorem \ref{thm:Milne_stability}. 
\end{rem}

\bibliographystyle{amsplain}

\end{document}